\documentclass[12pt]{amsart}

\usepackage[utf8]{inputenc}
\usepackage[T1]{fontenc}
\usepackage{amstext,amsthm,amsmath,amssymb}
\usepackage{mathrsfs}

\usepackage{enumitem}
\usepackage{euscript}

\usepackage{fouriernc}
\usepackage[scaled=0.875]{helvet} 
\usepackage{courier} 

\usepackage[normalem]{ulem}

\usepackage{wasysym}
\usepackage{marvosym}

\usepackage[babel=true]{csquotes}

\usepackage[marginparwidth=2cm]{geometry}
\usepackage[usenames,dvipsnames,svgnames,table]{xcolor}
\usepackage[colorlinks=true,urlcolor=DarkBlue,linkcolor=DarkRed,citecolor=DarkGreen,pagebackref=true]{hyperref}

\usepackage{tikz}
\usetikzlibrary{shapes,arrows}
\usetikzlibrary{fit,positioning}
\usetikzlibrary{patterns}
\usepackage{xkeyval}
\usepackage{moreverb}
\usepackage{epic}
\usepgfmodule{shapes,plot,decorations}

\usepackage{marginnote}
\usepackage{epsfig}
\usepackage{pinlabel}
\usepackage{subfig}

\usepackage{nicefrac}

\usepackage{ifmtarg}

\newcommand{\Z}{\mathbb{Z}}

\newcommand{\R}{\mathbb{R}}

\newcommand{\Hb}{\mathbb{H}}

\newcommand{\Pb}{\mathbb{P}}
\newcommand{\Sb}{\mathbb{S}}

\newcommand{\G}{\Gamma}
\newcommand{\g}{\gamma}

\newcommand{\Cc}{\mathcal{C}}

\newcommand{\Fc}{\mathcal{F}}

\newcommand{\Hc}{\mathcal{H}}

\newcommand{\Vc}{\mathcal{V}}

\renewcommand{\Hc}{\mathcal{H}}
\newcommand{\Pc}{\mathcal{P}}
\newcommand{\Tc}{\mathcal{T}}

\renewcommand{\O}{\Omega}

\renewcommand{\leq}{\leqslant}
\renewcommand{\geq}{\geqslant}

\theoremstyle{plain}
\newtheorem{theorem}{Theorem}[section]
\newtheorem{question}[theorem]{Question}
\newtheorem{proposition}[theorem]{Proposition}
\newtheorem{corollary}[theorem]{Corollary}
\newtheorem{lemma}[theorem]{Lemma}
\newtheorem{fact}[theorem]{Fact}

\newtheorem{theointro}{Theorem}[section]

\theoremstyle{definition}
\newtheorem{definition}[theorem]{Definition}

\theoremstyle{remark}
\newtheorem{remark}[theorem]{Remark}

\title[A small closed convex projective 4-manifold via Dehn filling]{A small closed convex projective \\ 4-manifold via Dehn filling}

\author{Gye-Seon Lee}
\address{Department of Mathematics, Sungkyunkwan University, Suwon, South Korea}
\email{gyeseonlee@skku.edu}

\author{Ludovic Marquis}
\address{Univ Rennes, CNRS, IRMAR - UMR 6625, F-35000 Rennes, France}
\email{ludovic.marquis@univ-rennes1.fr}

\author{Stefano Riolo}
\address{Dipartimento di Matematica, Universit\`a di Pisa, Italy}
\email{stefano.riolo@dm.unipi.it}


\begin{document}

\begin{abstract}
In order to obtain a closed orientable convex projective four-manifold with small positive Euler characteristic, we build an explicit example of convex projective Dehn filling of a cusped hyperbolic four-manifold through a continuous path of projective cone-manifolds.
\end{abstract}


\maketitle

\section{Introduction} 

Convex projective manifolds form an interesting class of aspherical manifolds, including complete hyperbolic manifolds. We refer to \cite{Bsurvey,survey_ludo,survey_CLM} and \cite{Msurvey} for surveys on convex projective manifolds and hyperbolic 4-manifolds, respectively. This class of geometric manifolds has been studied notably in the context of deformations of geometric structures on manifolds or orbifolds (see the survey \cite{survey_CLM} and the references therein), or for its link to dynamical systems through the notion of Anosov representation \cite{convexe_div_1,CC_Opq,DGK17} (see \cite{labourie_anosov,GW12} for the notion of Anosov representation).

\medskip

A \emph{convex projective $n$-manifold} is the quotient $\Omega/_\Gamma$ of  a properly convex\footnote{A subset $\Omega$ of $\R\Pb^n$ is \emph{properly convex} if its closure $\overline{\Omega}$ is contained and convex in some affine chart.} domain $\Omega$ in the real projective space $\R\Pb^n$ by a subgroup $\Gamma$ of the projective linear group $\mathrm{PGL}_{n+1}\R$ acting freely and properly discontinuously on $\Omega$. A \emph{convex projective $n$-orbifold} is defined similarly without requiring the action of $\Gamma$ to be free. If $\Omega$ is endowed with its Hilbert metric, then $\Gamma$ acts on $\Omega$ by isometry, and the manifold (or orbifold) $\Omega/_\Gamma$ inherits a complete Finsler metric (see \cite{intro_constantin} for an introduction to Hilbert geometry). In the case that $\Omega$ is an open ellipsoid, it is isometric to the hyperbolic space $\Hb^n$,\footnote{This is in fact the projective model of the hyperbolic space, also known as the Beltrami--Cayley--Klein model.} and the quotient $\Omega/_\Gamma$ is a complete hyperbolic manifold (or orbifold). A complete hyperbolic manifold is \emph{cusped} if it is non-compact and of finite volume. A \emph{convex projective} (resp. \emph{complete hyperbolic}) \emph{structure} on a manifold $M$ is a diffeomorphism between $M$ and a convex projective (resp. hyperbolic) manifold $\Omega/_\Gamma$.

\medskip

The goal of this paper is to prove the following:

\begin{theointro} \label{thm:main}
There exists a closed orientable convex projective $4$-manifold $X$ containing 10 disjoint totally geodesic $2$-tori $\Sigma=T_1\sqcup \ldots \sqcup T_{10} \subset X$ 
such that:
\begin{enumerate}
\item\label{item:thmA_hyperbolic}  The complement $M=X\smallsetminus\Sigma$ admits a complete finite-volume hyperbolic structure.
\item\label{item:thmA_Euler} The Euler characteristic of $X$ (and of $M$) is $12$.
\item\label{item:thmA_curve} The hyperbolic manifold $M$ has a maximal cusp section in which each filling curve has length $6$.
\item\label{item:thmA_rel} The fundamental group $\pi_1 X$ is relatively hyperbolic with respect to the collection of rank-$2$ abelian subgroups $\{\pi_1T_i,\,\pi_1T'_i\}_{i}$, where $\{T'_1,\ldots,T'_{10}\}$ is another collection of disjoint, totally geodesic, $2$-tori such that each $T_i$ is transverse to each $T'_j$.
\item\label{item:thmA_cone} The hyperbolic structure $\sigma_0$ on $M$ and the convex projective structure $\sigma_{2\pi}$ on $X$ arise as limits of an analytic path $\theta\mapsto\sigma_\theta$
of projective cone-manifold structures on $X$, singular along $\Sigma$ with cone angle $\theta\in(0,2\pi)$.
\item\label{item:thmA_orb} For each integer $m\geq1$, the structure $\sigma_{\nicefrac{2 \pi}{m}}$ is the underlying cone-manifold structure of a convex projective orbifold $\Omega_m/_{\Gamma_m}$. For $m\geq2$, the group $\Gamma_m$ is relatively hyperbolic with respect to the collection of rank-$2$ abelian subgroups $\{\pi_1T_i\}_{i}$.
\end{enumerate}
\end{theointro}

We refer the reader to Remark \ref{rem:nonsingular} for the meaning of totally geodesic submanifold in the real projective setting and  to Section \ref{sec:rel_hyp} for the definition and some facts on relative hyperbolicity.

\medskip

Note that the manifold $X$ does not admit a hyperbolic structure because $\pi_1 X$ contains $\Z^2$ (the tori $T_i$ and $T'_i$ are indeed $\pi_1$-injective). The hyperbolic manifold $M$ has 10 cusps, each with section a 3-torus. A \emph{filling curve} is a closed geodesic in a cusp section of $M$ (with respect to the induced flat metric) that bounds a disc in $X$.

\subsection*{Cone-manifolds and Dehn filling}

Projective cone-manifolds are singular projective manifolds generalising the more familiar hyperbolic cone-manifolds (see Definition \ref{def:cone-mfd}). The convex projective 4-manifold $X$ of Theorem \ref{thm:main} is obtained from the cusped hyperbolic 4-manifold $M$ by ``projective Dehn filling''. This is in analogy with Thurston's hyperbolic Dehn filling \cite{Tnotes}, where $\theta \mapsto\sigma_\theta$ is a path of hyperbolic cone-manifold structures on a 3-manifold $X$, singular along a link $\Sigma\subset X$. In both (hyperbolic and projective) cases, as the cone angle $\theta$ approaches $2\pi$, the projective cone-manifold structure becomes non-singular, and we get a convex projective structure on $X$. On the other extreme of the path, as $\theta$ tends to $0$, the singular locus $\Sigma$ is drilled away, giving rise to the cusps of the hyperbolic manifold $M$. 

\medskip

The projective cone-manifold structures $\sigma_\theta$ of Theorem \ref{thm:main} are singular along the tori $\Sigma$, and induce (non-singular) projective structures\footnote{A (\emph{real}) \emph{projective structure} on an $n$-manifold is a $(\mathrm{PGL}_{n+1}\R,\R\Pb^n)$-structure.}
on both $M=X\smallsetminus\Sigma$ and $\Sigma$.
The path $\theta\mapsto\sigma_\theta$ is \emph{analytic}, meaning that for each $\theta \in (0,2\pi)$, it is possible to choose a holonomy representation $\rho_\theta\in\mathrm{Hom}(\pi_1M,\mathrm{PGL}_5 \R)$ of the projective structure $\sigma_\theta|_M$ so that the function $\theta\mapsto\rho_\theta(\gamma)$ is analytic for all $\gamma \in \pi_1M$. Each torus $T_i \subset \Sigma$ has a meridian $\gamma_i\in\pi_1M$ whose holonomy $\rho_\theta(\gamma_i)\in\mathrm{PGL}_5 \R$ is conjugate to a (projective) rotation of angle $\theta$. In addition, the sequence of representations $\{\rho_{\nicefrac{2\pi}{m}}\}_m$ converges algebraically to $\rho_0$ as $m \rightarrow \infty$, and the sequence of convex sets $\{\overline{\Omega_m}\}_m$ converges to $\overline{\Hb^4}\subset\R\Pb^4$ in the Hausdorff topology.\footnote{These additional facts can be proved as in \cite[\S 12]{CLM}.}

\subsection*{New features of the result}

Theorem \ref{thm:main} is shown by an explicit construction. Since the convex projective manifold $X$ has non-zero Euler characteristic, it is indecomposable\footnote{A properly convex domain $\Omega$ of $\R\Pb^n$ is \emph{indecomposable} if it is not a convex hull of lower-dimensional domains. A convex projective manifold or orbifold $\Omega/_\Gamma$ is \emph{indecomposable} if $\Omega$ is indecomposable.} (see Fact \ref{fact:indecomposable}). It seems that, at the time of writing this paper, the literature misses concrete examples of closed indecomposable convex projective $n$-manifolds, $n\geq4$, which do not admit a hyperbolic structure.
For the moment, we know only two techniques to obtain such manifolds:
\begin{enumerate}
\item torsion-free subgroups of some discrete projective reflection groups using Vinberg’s theory \cite{V}, as shown by Benoist \cite{CD4,Benoist_quasi} and Choi and the first two authors \cite{CLM,CLM_ecima};
\item some Gromov--Thurston manifolds, as shown by Kapovich \cite{K}.
\end{enumerate}

In contrast with Theorem \ref{thm:main}, the Selberg lemma has an important role to guarantee the existence of such manifolds in all these cases. In particular, very little is known about the topology of closed convex projective manifolds. Note that the techniques involved in the construction of our $X$ are in the spirit of (1) rather than (2).

\begin{remark}
There is a clear distinction between the manifolds constructed in \cite{Benoist_quasi,K} and the ones in \cite{CD4,CLM,CLM_ecima}, including our $X$: the fundamental groups of the former are Gromov-hyperbolic, but those of the latter are not.
\end{remark}

The Euler characteristic of a closed even-dimensional manifold can be seen as a rough measure of its topological complexity. Note that a well-known conjecture states that closed aspherical 4-manifolds have Euler characteristic $\chi\geq0$, and this is certainly true in the hyperbolic case by the Gau{\ss}--Bonnet theorem. Our manifold $X$ has $\chi(X)=12$, and appears to be the closed orientable indecomposable convex projective 4-manifold with the smallest known Euler characteristic (to the best of our knowledge). In the hyperbolic case, the smallest known value of $\chi$ is $16$ \cite{CM,L}.\footnote{
Similarly to the orientable hyperbolic 4-manifolds that one gets from \cite{CM,L}, our manifold $X$ is built as the orientable double cover of a non-orientable convex projective manifold.
So the smallest known value of $\chi$ for a closed indecomposable convex projective (resp. hyperbolic) 4-manifold is currently $6$ (resp. $8$).}

\medskip

Theorem \ref{thm:main} is an effective version, in dimension four, of a result by Choi and the first two authors \cite[Theorem B]{CLM}. Let us first recall their construction, called ``convex projective generalised Dehn filling''. They build a sequence of discrete projective reflection groups $\{\G_m\}_{m\geq m_0}$ of $\mathrm{PGL}_{n+1}\R$, each acting cocompactly on a properly convex domain $\Omega_m\subset\R\Pb^n$ of dimension $n=4, 5$ or $6$, whose limit as $m\to\infty$ is a discrete hyperbolic reflection group $\G_\infty<\mathrm{Isom}(\Hb^n)$ of finite covolume. A fundamental domain of the group $\Gamma_m$ is a compact Coxeter polytope $P_m$ in $\Omega_m$ whose combinatorics does not depend on $m$. The hyperbolic Coxeter polytope $P_\infty$, instead, is combinatorially obtained from $P_m$ by substituting a ridge with an ideal vertex. In other words, the cusp of the hyperbolic orbifold $\Hb^n/_{\Gamma_\infty}$ is ``projectively filled''. By applying a refined version of Selberg’s lemma to $\Gamma_\infty$, they get a statement similar to Theorem \ref{thm:main}.\eqref{item:thmA_orb}. The difference is that $X=\O_{m_0}/_{\G'_{m_0}}$ is ``only'' an orbifold (where $\G'_{m_0}$ is a finite-index subgroup of $\G_{m_0}$). To promote $X$ to a manifold, one should then apply again the Selberg lemma, this time to $\G'_{m_0}$. Thus, our improvement is two-fold: we found an $X$ with small Euler characteristic, and a continuous (rather than discrete) family of cone-manifolds. 

\medskip

Another interesting feature of $X$ is the relative hyperbolicity of $\pi_1X$. Indeed, Gromov and Thurston \cite{bleiler_hodgson,A}\footnote{A proof of the Gromov--Thurston $2\pi$ theorem was given by Bleiler and Hodgson \cite{bleiler_hodgson} in the context of $3$-manifolds, and the same proof holds in any dimension, as explained in \cite[Section 2.1]{A}.} have shown that the fundamental group of a Dehn filling of a hyperbolic manifold with torus cusps is relatively hyperbolic with respect to the subgroups associated to the inserted tori, provided that the filling curves are longer than $2\pi$.\footnote{See also \cite{osin,groves_manning,FM2} for the geometric group theoretic generalisation of that statement.} The fundamental group of $X$ is not relatively hyperbolic with respect to the subgroups associated to the inserted tori by Theorem \ref{thm:main}.\eqref{item:thmA_rel} (see Section \ref{sec:rel_hyp}).\footnote{It is relatively hyperbolic with respect to a larger family of abelian groups.} There is no contradiction between the Gromov--Thurston $2\pi$ theorem and Theorem \ref{thm:main}.\eqref{item:thmA_curve} because $2 \pi > 6$.

\subsection*{Divisible convex domains} 

Recall that a properly convex domain $\Omega$ of $\R\Pb^n$ is \emph{divisible \emph{(}by $\Gamma$\emph{)}} if there exists a discrete subgroup $\Gamma$ of $\mathrm{PGL}_{n+1}\R$ acting cocompactly on $\Omega$. A theorem of Benoist \cite{convexe_div_1} implies that the indecomposable divisible convex domains $\Omega_m\subset\R\Pb^4$ of Theorem \ref{thm:main}.(\ref{item:thmA_orb}) are not strictly convex\footnote{A subset $\Omega$ of $\R\Pb^n$ is \emph{strictly convex} if it is properly convex and its boundary does not contain any non-trivial projective line segment.} because the groups $\Gamma_m$ of Theorem \ref{thm:main}.(\ref{item:thmA_orb}) are not Gromov-hyperbolic.

\medskip

There are very few currently known constructions of inhomogeneous indecomposable divisible non-strictly convex domains. For a complete historical account, we refer the reader to the introduction of \cite{CLM}. Here we mention only its essentials.

\medskip

The first construction of such domains is due to Benoist \cite{CD4}, and has been extended in \cite{ecima_ludo,BDL_3d_convex,CLM_ecima}. In those constructions the compact quotient $\Omega/_\G$ is homeomorphic to the union along the boundaries of finitely many submanifolds, each admitting a complete finite-volume hyperbolic structure on its interior. As a result, if $\O$ is of dimension $n$, then $\G$ is relatively hyperbolic with respect to a collection of virtually abelian subgroups of rank $n-1$.

\medskip

In \cite{CLM}, a different construction of inhomogeneous indecomposable divisible non-strictly convex domains is given by convex projective generalised Dehn filling.
In contrast with the previous examples, these are relatively hyperbolic with respect to a collection of virtually abelian subgroups of rank $n-2$.
The divisible (by $\Gamma_m$) domains $\O_m\subset\R\Pb^4$ of Theorem \ref{thm:main}.(\ref{item:thmA_orb}) are new examples of this kind.

\medskip

We point out that at the time of writing there is no example of inhomogeneous indecomposable divisible non-strictly convex domain of dimension $n$, for any $n \geqslant 9$.

\medskip

We stress that such domains, to all appearances, are linked to the geometrisation problem, i.e. putting a $(G,X)$-structure on a manifold. So far, almost all manifolds geometrised through this process are either obtained by gluing cusped hyperbolic manifolds, or by Dehn filling of a cusped hyperbolic manifold. Here, the goal is to do so with a small manifold and using cone-manifolds. It is especially important that we do not use Selberg's lemma.

\subsection*{Dehn fillings of hyperbolic manifolds}

Let us say that a closed manifold $X$ is a \emph{filling} of a manifold $M$ if there exists a codimension-2 submanifold $\Sigma\subset X$ such that the complement $X\smallsetminus\Sigma$ is diffeomorphic to $M$.
Note that the manifold $M$ is diffeomorphic to the interior of a compact manifold $\overline M$ whose boundary $\partial\overline M$ fibres in circles over $\Sigma$. Given $M$ as above, we obtain a filling by attaching to $\overline M$ the total space of a $D^2$-bundle $E\to\Sigma$ through a diffeomorphism $\partial\overline M\to\partial E$. This operation is commonly called a \emph{Dehn filling} of $M$. Any cusped hyperbolic manifold $M$ has a finite covering $M'$ with torus cusps (see e.g. \cite[Theorem 3.1]{MRS}). In other words, $\partial\overline{M'}$ consists of $(n-1)$-tori. The manifold $M'$ has typically infinitely many fillings up to diffeomorphism. 

\medskip

Thurston's hyperbolic Dehn filling theorem states that every filling of a cusped hyperbolic 3-manifold with torus cusps, except for finitely many fillings on each cusp, admits a hyperbolic structure. In dimension $n\geq4$, except for finitely many fillings on each cusp, the fundamental group of a filling of a cusped hyperbolic $n$-manifold is relatively hyperbolic with respect to a collection of subgroups virtually isomorphic to $\Z^{n-2}$, by \cite[Theorem~1.1]{osin}. Since $n\geqslant 4$, those groups contains $\Z^2$ and so those fillings do not admit any hyperbolic structure. The geometry of the remaining fillings is rather unpredictable, but it is expected that they also do not carry a hyperbolic structure. But, Theorem \ref{thm:main}.(\ref{item:thmA_hyperbolic}) and \cite[Theorem B]{CLM} show that some fillings of some cusped hyperbolic $n$-manifolds  admit a convex projective structure. This leads to the following:

\begin{question} \label{quest:filling}
Which filling of a cusped hyperbolic manifold of dimension $n\geq4$ (with torus cusps) admits a convex projective structure?
\end{question}

It is worth mentioning that almost all fillings of any cusped hyperbolic manifold with torus cusps admit a complete Riemannian metric of non-positive sectional curvature by the Gromov--Thurston $2\pi$ theorem \cite{bleiler_hodgson,A}, and an Einstein metric of negative scalar curvature by work of Anderson \cite{A} and Bamler \cite{B}. Both, in some sense, extend Thurston's 3-dimensional theorem to higher dimension (compare also with \cite{S_cuspclosing,FM2,FM}). But those theorems cannot be applied to the manifold $X$ of Theorem \ref{thm:main}, since the filling curves are too short.

\medskip

Let us also note that there is an opportune version of hyperbolic Dehn filling in dimension 4: one can sometimes fill some cusps of a hyperbolic 4-manifold and get another cusped hyperbolic 4-manifold, at the expense of drilling some totally geodesic surfaces \cite{MR,LMRY}.

\subsection*{Projective flexibility}

Thurston's hyperbolic Dehn filling theorem essentially relies on the flexibility of the complete hyperbolic structure of cusped hyperbolic 3-manifolds. The local rigidity theorem of Garland and Raghunathan \cite{GR} (see also \cite{MR2110758}), on the other hand, says that the holonomy representation $\rho$ of any cusped hyperbolic manifold $M$ of dimension $n\geq4$ has a neighbourhood in $\mathrm{Hom}(\pi_1M,\mathrm{Isom}(\Hb^n))$ consisting of conjugates of $\rho$. Now, Theorem \ref{thm:main}.(\ref{item:thmA_cone}) shows that the hyperbolic structure $\sigma_0$ on the 4-manifold $M$ is \emph{projectively flexible}, i.e. the conjugacy class of the holonomy representation $\rho_0$ of $\sigma_0$ $$[\rho_0]\in\mathrm{Hom}(\pi_1M,\mathrm{PGL}_5\R)\big/_{\mathrm{PGL}_5\R}$$
is not an isolated point. To avoid confusion with the terminology, we mention that in this definition of flexibility there is no restriction on the holonomy of the peripheral subgroups of $\pi_1M$. For instance, all cusped hyperbolic 3-manifolds are projectively flexible by the hyperbolic Dehn filling theorem. It is thus natural to ask the following question, which is a priori different from Question \ref{quest:filling}.

\begin{question} \label{quest:flexibility}
Which cusped hyperbolic 4-manifold is projectively flexible?
Is every cusped hyperbolic 4-manifold finitely covered by a projectively flexible one?\footnote{Similar considerations were 
done by Cooper, Long and Thistlethwaite \cite{CLT07} for closed hyperbolic 3-manifolds. }
\end{question}

We note that some cusped hyperbolic 4-orbifolds, for example the Coxeter pyramid $[6,3,3,3,\infty]$, are not projectively flexible \cite[Proposition 20]{V}. 

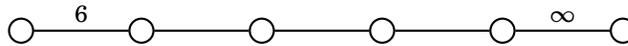
\begin{figure}[h!]
\begin{center}
\begin{tikzpicture}[thick,scale=0.8, every node/.style={transform shape}]
\node[draw, circle, inner sep=1pt, minimum size=4mm] (1) at (0,0)  {};
\node[draw, circle, inner sep=1pt, minimum size=4mm] (2) at (2,0)  {};
\node[draw, circle, inner sep=1pt, minimum size=4mm] (3) at (4,0)  {};
\node[draw, circle, inner sep=1pt, minimum size=4mm] (4) at (6,0)  {};
\node[draw, circle, inner sep=1pt, minimum size=4mm] (5) at (8,0)  {};
\node[draw, circle, inner sep=1pt, minimum size=4mm] (6) at (10,0)  {};

\draw (1) -- (2) node[above,midway]   {$6$};
\draw (2) -- (3) node[above,midway]   {};
\draw (3) -- (4) node[above,midway]   {};
\draw (4) -- (5) node[above,midway]   {};
\draw (5) -- (6) node[above,midway]   {$\infty$};
\end{tikzpicture}
\end{center}
\caption{\footnotesize The Coxeter diagram of the pyramid $[6,3,3,3,\infty]$ (see Section \ref{subsec:Coxeter_groups} for the basic terminology on Coxeter groups). The associated cusped hyperbolic 4-orbifold is projectively rigid.}
\label{fig:Coxeter_pyramid}
\end{figure}

\subsection*{On the proof}

As already said, the proof of Theorem \ref{thm:main} is constructive. We begin with the ideal hyperbolic rectified 4-simplex $R\subset\Hb^4$, which is a Coxeter polytope. By applying the techniques introduced in \cite{CLM}, we perform a ``convex projective generalised Dehn filling'' to $R$: the hyperbolic structure on the orbifold $R$ is deformed to projective structures which extend to structures of ``mirror polytope'' (see Section \ref{sec:gen_dehn_fill}) on the bitruncated 4-simplex $Q$ (see Section \ref{sec:truncated_simplex}). Note that $Q$ minus some ridges is stratum-preserving homeomorphic to $R$. This will be translated into the fact that $X$ minus some tori is homeomorphic to $M$.

\medskip

We build the hyperbolic manifold $M$ as an orbifold covering of $R$, by exploiting a construction by Kolpakov and Slavich \cite{KS}. By lifting the deformation from $R$ to $M$, we get the path $\theta\mapsto\sigma_\theta$. The manifold $X$ covers a Coxeter orbifold based on the bitruncated 4-simplex $Q$.

\medskip

Since there exist cusped orientable hyperbolic 4-manifolds $M_0$ tessellated by copies of $R$ with $\chi(M_0)<12$ \cite{KS,Sl,RiSl,KRR}, one could wonder why not to build a smaller convex projective manifold $X$ by Dehn filling such an $M_0$. A first obstruction is topological: a cusp section of those $M_0$ does not always fibre in circles. Even when all cusp sections of such an $M_0$ do fibre in circles, $M_0$ does not cover $R$ but covers the quotient of $R$ by its symmetry group. The latter is the Coxeter pyramid $[6,3,3,3,\infty]$ \cite[Lemma 2.2]{RiSl}, which is projectively rigid, so our technique does not apply in these cases.

\subsection*{Structure of the paper}

In Section \ref{sec:preliminaries} we introduce some basic concepts of projective cone-manifolds, mirror polytopes and the truncation process of the 4-simplex, and in Section \ref{sec:proof} we prove Theorem \ref{thm:main}.

\subsection*{Acknowledgements}
We would like to thank Caterina Campagnolo, Roberto Frigerio, Jason Manning, Marco Moraschini and Andrea Seppi for interesting discussions. Part of this work was done when the second and third authors visited the mathematics department of Heidelberg university, which we thank for the hospitality. We also thank the referees for their time and their valuable advice and comments.

\smallskip

{\small G.L. was supported by the European Research Council under ERC-Consolidator Grant 614733 and by DFG grant LE 3901/1-1 within the Priority Programme SPP 2026 “Geometry at Infinity”, and he acknowledges support from U.S. National Science Foundation grants DMS 1107452, 1107263, 1107367 “RNMS: Geometric structures And Representation varieties” (the GEAR Network).
L.M. acknowledges the Centre Henri Lebesgue ANR-11-LABX-0020 LEBESGUE for its support. S.R. was supported by the Mathematics Department of the University of Pisa (research fellowship “Deformazioni di strutture iperboliche in dimensione quattro”), and by the Swiss National Science Foundation (project no.~PP00P2-170560).}

\section{Preliminaries} \label{sec:preliminaries}

In this section we introduce some preliminary notions and fix some notation.

\subsection{Projective cone-manifolds} \label{sec:cone-mfds}

Riemannian $(G,X)$ cone-manifolds were introduced by Thurston \cite{Tshapes} (see also \cite{McM}). If the geometry $(G,X)$ locally embeds in real projective geometry $(\mathrm{PGL}_{n+1}\R,\R\Pb^n)$ (such as the constant-curvature geometries), a $(G,X)$ cone-manifold can be thought as a projective cone-manifold. Hyperbolic cone-manifolds of dimension 3 appear in the proofs of Thurston's hyperbolic Dehn filling theorem \cite{Tnotes} (see also \cite[Chapter 15]{Mbook}) and of the orbifold theorem \cite{CHK,BLP}.
Projective cone-manifolds were introduced by Danciger \cite{D1,D2} in the context of  geometric transition from hyperbolic to Anti-de Sitter 3-dimensional structures. Quite recently, some higher-dimensional cone-manifolds are used, in particular, in dimension 4: for hyperbolic Dehn filling or degeneration \cite{MR,LMRY}, and in the projective context of AdS-hyperbolic transition \cite{RS}.

\medskip

We now define projective cone-manifolds ``with cone angles along link singularities''. 
Our definition is in the spirit of Barbot--Bonsante--Schlenker \cite{BBS}. 

\medskip

Let $\Sb^n=\left(\R^{n+1}\smallsetminus\{0\}\right)/_{\R_{>0}}$ be the projective sphere and $\hat{\Sb} \colon\R^{n+1}\smallsetminus\{0\}\to\Sb^n$ the canonical projection. For every subset $U$ of $\R^{n+1}$, let $\Sb(U)$ denote $\hat{\Sb}(U \smallsetminus \{0\})$. With a little abuse of notation, we embed the projective spheres $\Sb^{n-2}$ and $\Sb^1$ into $\Sb^n$, $n\geq2$, as follows:
$$\Sb^{n-2}=\Sb(\{(x_1,\ldots,x_{n-1},0,0)\})\subset\Sb^n\mbox{\ \ \ and\ \ \ }\Sb^1=\Sb(\{(0,\ldots,0,x_n,x_{n+1})\})\subset\Sb^n.$$ 
Given an open subset $U\subset\Sb^1$, we define
$$\Sb^{n-2}*U\ =\bigcup_{p\in\Sb^{n-2}}\ \bigcup_{q\in U}\ [p,-p]_q\ \subset\ \Sb^n,$$
where $[p,-p]_q\subset\Sb^n$ denotes the half circle containing $q$ with endpoints $p$ and its antipode $-p$. For example, $\Sb^{n-2}*\Sb^1 = \Sb^n$.

\medskip

A \emph{projective circle} is a closed connected $(\mathrm{SL}_2\R,\Sb^1)$-manifold. Let $C$ be a projective circle. We may think of $C$ as
$$C= \Big(\coprod_{\alpha \in A }\, U_{\alpha}\Big) \Big/_\sim$$
for a collection $\{ U_{\alpha} \}_{\alpha \in A}$ of open subsets of $\Sb^1$, a collection $\{ U_{\alpha\beta} \}_{\alpha,\beta \in A}$ of open subsets $U_{\alpha\beta} \subset U_\alpha$, a collection $\{ g_{\alpha \beta} \}_{\alpha,\beta \in A}$ of diffeomorphisms $g_{\alpha \beta} \colon U_{\alpha\beta} \rightarrow U_{\beta\alpha}$ which are restrictions of elements of $\mathrm{SL}_2\R$, and the relation $x\sim g_{\alpha\beta}(x)$ for all $x\in U_{\alpha\beta}$.

\medskip

We now add the extra requirement that $C=C_\theta$ is an \emph{elliptic circle}, i.e. that the holonomy representation $\rho$ of $C_\theta$ sends a generator $\gamma$ of $\pi_1 C_\theta$ to an elliptic element $\rho(\gamma) \in \mathrm{SL}_2\R$. Passing to the universal covering $\widetilde{\Sb^1}$ of $\Sb^1$ and the covering group $\widetilde{\mathrm{SL}_2\R}$ of $\mathrm{SL}_2\R$ which acts on $\widetilde{\Sb^1}$, we lift $\rho$ to a representation $\widetilde{\rho} \colon \pi_1 C_\theta \rightarrow \widetilde{\mathrm{SL}_2\R}$. To the element $\widetilde{\rho}(\gamma) \in \widetilde{\mathrm{SL}_2\R}$ is naturally associated a unique real number $\theta > 0$ which characterises the elliptic circle $C_\theta$. Note that $\rho(\gamma)$ is conjugate to $\left(\begin{smallmatrix}
\cos\theta & -\sin\theta \\
\sin\theta & \phantom{-}\cos\theta
\end{smallmatrix}\right)$ (see \cite[Section 5.4]{Gol18}).

\medskip

By extending $g_{\alpha\beta}\in\mathrm{SL}_2\R$ to $\hat{g}_{\alpha\beta} = \left(\begin{smallmatrix}
\mathrm{Id} & 0 \\
0 & g_{\alpha\beta}
\end{smallmatrix}\right) \in\mathrm{SL}_{n+1}\R$ (thus fixing $\Sb^{n-2}\subset\Sb^n$ pointwise), 
we can define
$$\Sb^{n-2}*C_\theta = \Big(\coprod_{\alpha \in A}\, \Sb^{n-2}*U_{\alpha} \Big)\Big/_\sim$$
with the relation $x\sim \hat{g}_{\alpha\beta}(x)$ for all $x\in \Sb^{n-2}*U_{\alpha\beta}$.

\medskip

This space will be the local model for our cone-manifolds. By canonically embedding $\Sb^{n-2}$ and $C_\theta$ into $\Sb^{n-2}*C_\theta$, we have that the couple $(\Sb^{n-2}*C_\theta,\Sb^{n-2})$ is homeomorphic to $(\Sb^n,\Sb^{n-2})$. Moreover, the sphere $\Sb^{n-2}*C_\theta$ is stratified in two projective manifolds:
\begin{itemize}
\item the \emph{singular locus} $\Sb^{n-2}$, and
\item the \emph{regular locus} $(\Sb^{n-2}*C_\theta)\smallsetminus\Sb^{n-2}$.
\end{itemize}
The holonomy of a meridian\footnote{Let $N$ be a manifold and $S\subset N$ a connected submanifold of codimension 2. We call \emph{meridian} of $S$ an element $\gamma\in\pi_1(N\smallsetminus S)$ that is represented by a curve which is freely homotopic in $N\smallsetminus S$ to the boundary of a fibre of a tubular neighbourhood of $S$ in $N$.} of the singular locus in the regular locus is the holonomy of a generator of $\pi_1 C_\theta$.

\medskip

\begin{definition} \label{def:cone-mfd}
Let $X$ be an $n$-manifold and $\Sigma\subset X$ a codimension-2 submanifold. A \emph{projective cone-manifold structure} on $X$, singular along $\Sigma$ \emph{with cone angles}, is an atlas for
$X$ whose each chart has values into some $\Sb^{n-2}*C_\theta$ and sends the points of $\Sigma$ to the singular locus $\Sb^{n-2}$, and whose transition functions restrict to isomorphisms of projective manifolds on each stratum.
\end{definition}

The \emph{regular locus} of the cone-manifold $X$ is the complement $X\smallsetminus\Sigma$, while $\Sigma$ is the \emph{singular locus}. Both are (non-singular) projective manifolds. To each connected component $\Sigma_i$ of $\Sigma$ is associated a cone angle $\theta_i>0$.

\begin{remark} \label{rem:nonsingular}
A projective cone-manifold $X$ with all cone angles $\theta=2\pi$ is simply a projective manifold with a totally geodesic codimension-2 submanifold $\Sigma\subset X$. Here, by \emph{totally geodesic}, we mean that the preimage of $\Sigma$ under the universal covering of $X$ is locally mapped to codimension-2 projective subspaces of $\Sb^n$. If moreover $X$ is a convex projective (and so Finsler) manifold, then $\Sigma$ is totally geodesic in the usual sense.
\end{remark}

\subsection{Mirror polytopes} \label{sec:gen_dehn_fill}

We now introduce the main objects for our proof of Theorem \ref{thm:main}: mirror polytopes. Roughly speaking, a mirror polytope is a polytope in the projective sphere, together with a choice of a projective reflection along each of the supporting hyperplanes of the facets (satisfying some extra conditions). We refer to \cite{CLM,V} for further details. 

\medskip

A subset $P$ of $\Sb^n$ is \emph{convex} if there exists a convex cone\footnote{By a \emph{cone}, we mean a subset of $\R^{n+1}$ which is invariant under multiplication by positive scalars.} $U$ of $\R^{n+1}$ such that $P = \mathbb{S}(U)$, and moreover a convex subset $P$ is \emph{properly convex} if its closure $\overline{P}$ does not contain a pair of antipodal points. A \emph{projective $n$-polytope} is a properly convex subset $P$ of $\Sb^n$ such that $P$ has a non-empty interior and $$P = \bigcap_{i=1}^{N} \Sb(\{ v \in \R^{n+1} \mid \alpha_i(v) \leqslant 0 \})$$
where $\alpha_i$, $i=1, \dotsc, N$, are linear forms on $\R^{n+1}$. We always assume that the set of linear forms is \emph{minimal}, i.e. none of the half-spaces $\Sb(\{ v \in \R^{n+1} \mid \alpha_i(v) \leqslant 0 \})$ contains the intersection of all the others, hence the polytope $P$ has $N$ facets. A \emph{facet} (resp. \emph{ridge}) of a polytope is a face of codimension $1$ (resp. $2$). 

\medskip

A \emph{projective reflection} is an element of $\mathrm{SL}^{\pm}_{n+1}\R$ of order 2 which is the identity on a projective hyperplane. Each projective reflection $\sigma$ can be written as:
$$\sigma=\mathrm{Id}-\alpha\otimes b,\quad \textrm{in other words} \quad \sigma(v) = v - \alpha(v) b \quad \forall v \in \R^{n+1},$$
where $\alpha$ is a linear form on $\R^{n+1}$ and $b$ is a vector in $\R^{n+1}$ such that $\alpha(b)=2$. The projective hyperplane $\Sb (\ker(\alpha))$ is called the \emph{support} of $\sigma$. A \emph{projective rotation} is an element of $\mathrm{SL}_{n+1}\R$ which is the identity on a subspace $H \subset \R^{n+1}$ of codimension 2 and whose induced linear map from $\R^{n+1}/_H$ to itself is conjugate to a matrix 
$\left(\begin{smallmatrix}
\cos\theta & -\sin\theta \\
\sin\theta & \phantom{-}\cos\theta
\end{smallmatrix}\right)$
with $0 < \theta < 2 \pi$. The real number $\theta$ is called the \emph{angle} of rotation.

\begin{definition}\label{def:mirror_poly}
A \emph{mirror polytope} is a pair consisting of a projective polytope $P$ of $\Sb^n$ and a finite collection of projective reflections $\{ \sigma_s = \mathrm{Id} - \alpha_s \otimes b_s \}_{s \in S}$ with $\alpha_s(b_s)=2$ such that: 
\begin{itemize}
\item The index set $S$ consists of all the facets of $P$.
\item For each facet $s \in S$, the support of $\sigma_s$ is the supporting hyperplane of $s$.
\item For every pair of distinct facets $s,t$ of $P$,
$\alpha_s(b_t)$ and $\alpha_t(b_s)$ are either both negative or both zero.
\end{itemize}
\end{definition}

\begin{remark}
The third item of Definition~\ref{def:mirror_poly} may seem a bit awkward at first glance. In fact, \cite[Proposition~6]{V} shows that the third item holds when the group $\Gamma$ generated by $\{ \sigma_s \}_{s \in S}$ satisfies the condition:
$$
\gamma\mathrm{Int}(P) \cap \mathrm{Int} (P) = \varnothing, \quad \forall \gamma \in \Gamma \smallsetminus\{ \mathrm{Id} \},
$$
where $\mathrm{Int}(P)$ denotes the interior of $P$.
\end{remark}

Note that the reflections $\sigma_s$ of $P$ determine the couples $\{ (\alpha_s,b_s)  \}_{s \in S}$ up to $\alpha_s \mapsto \lambda_s \alpha_s$ and $b_s \mapsto \lambda_s^{-1} b_s$ for $\lambda_s >0$, because $P$ is defined by $\alpha_i \leqslant 0$. Moreover, $\alpha_s(b_t) \alpha_t(b_s) = 4 \cos^2 \theta$ for some $\theta \in (0,\nicefrac{\pi}{2}]$ if and only if the product $\sigma_s \sigma_t$ is a rotation of angle $2\theta$. The \emph{dihedral angle} of a ridge $s \cap t$ of a mirror polytope is $\theta$ if $\sigma_s \sigma_t$ is a rotation of angle $2\theta$, and $0$ otherwise. A \emph{Coxeter polytope} is a mirror polytope whose dihedral angles are submultiples of $\pi$, i.e. each dihedral angle is $\nicefrac{\pi}{m}$ for some integer $m \geqslant 2$ or $m=\infty$.

\medskip

For any number $N \in \mathbb{N}$, we set $[\![ N ]\!] := \{ 1, \dotsc, N\}$. An $N \times N$ matrix $A = (A_{ij})_{i, j \in  [\![ N ]\!]}$ is \emph{Cartan} if $(i)$ $A_{ii} = 2$ for all $i \in  [\![ N ]\!]$, $(ii)$ $A_{ij} \leqslant 0$, for all $i \neq j$, and $(iii)$ $A_{ij} =0 \Leftrightarrow A_{ji}=0 $, for all $i \neq j$. A Cartan matrix $A$ is \emph{irreducible} if it is not the direct sum of smaller matrices (after any reordering of the indices). Every Cartan matrix $A$ is obviously the direct sum of irreducible Cartan matrices $A'_1 \oplus \dotsm \oplus A'_k$. Each irreducible Cartan matrix $A'_i$, $i = 1, \dotsc, k$, is called a \emph{component} of $A$. If $x = (x_1, \dotsc, x_N)$ and $y = (y_1, \dotsc, y_N) \in \mathbb{R}^N$, we write $x > y$ if $x_i > y_i$ for all $i \in  [\![ N ]\!]$, and $x \geqslant y$ if $x_i \geqslant y_i$ for all $i \in  [\![ N ]\!]$.

\begin{proposition}[{Vinberg \cite[Theorem 3]{V}}]\label{prop:cartan_type}
If $A$ is an irreducible Cartan matrix of size $N \times N$, then exactly one of the following holds:
\begin{itemize}
\item[$\mathrm{(\!(P)\!)}$] $(i)$ The matrix $A$ is nonsingular, and $(ii)$ for every $x \in \R^N$, if $Ax \geqslant 0$, then $x>0$ or $x=0$.

\item[$\mathrm{(\!(Z)\!)}$] $(i)$ The rank of $A$ is $N-1$, $(ii)$ there exists a vector $y \in \R^N$ such that $y >0$ and $Ay =0$, and $(iii)$ for every $x \in \R^N$, if $Ax \geqslant 0$, then $Ax=0$.

\item[$\mathrm{(\!(N)\!)}$] $(i)$ There exists a vector $y \in \R^N$ such that $y >0$ and $Ay <0$, and $(ii)$ for every $x \in \R^N$, if $Ax \geqslant 0$ and $x \geqslant 0$, then $x=0$.
\end{itemize}
We say that $A$ is of \emph{positive, zero} or \emph{negative type} if $\mathrm{(\!(P)\!)}$, $\mathrm{(\!(Z)\!)}$ or $\mathrm{(\!(N)\!)}$ holds, respectively. 
\end{proposition}

A Cartan matrix $A$ is of \emph{positive} (resp. \emph{zero}) \emph{type} if every component of $A$ is of positive (resp. zero) type. The \emph{Cartan matrix} of a mirror polytope $P$ is the matrix $A_{P} = (\alpha_s(b_t))_{s,t \in S}$. Note that the Cartan matrix of $P$ is well-defined up to conjugation by a positive diagonal matrix because the reflections $\sigma_s$ of $P$ determine the couples $\{ (\alpha_s,b_s) \}_{s \in S}$ up to $\alpha_s \mapsto \lambda_s \alpha_s$ and $b_s \mapsto \lambda_s^{-1} b_s$ for $\lambda_s >0$. Two Cartan matrices $A$ and $B$ are \emph{equivalent} if $A = D B D^{-1}$ for some positive diagonal matrix $D$. We will make essential use of the following:

\begin{theorem}{\cite[Corollary 1]{V}}\label{thm:vinberg_unique}
Let $A$ be a Cartan matrix of size $N \times N$. If $A$ is irreducible, of negative type and of rank $n + 1$, then there exists a mirror polytope $P$ of $\Sb^{n}$ with $N$ facets (unique up to the action of $\mathrm{SL}^{\pm}_{n+1}\R$)  such that $A_P = A$.
\end{theorem}

\begin{remark}
Theorem \ref{thm:vinberg_unique} is not explicitly stated in \cite[Corollary 1]{V} for non-Coxeter polytopes, but it follows from Propositions 13 and 15 of \cite{V} that the consequent Corollary 1 of \cite{V} is valid not only for Coxeter polytopes, but also for mirror polytopes.
\end{remark}

To understand the combinatorics\footnote{The \emph{combinatorics} (or \emph{face poset}) of a polytope is the poset of its faces partially ordered by inclusion.} of a mirror polytope $P$ with facets $\{s\}_{s \in S}$, we introduce the poset $\sigma(P) \subset 2^{S}$ partially ordered by inclusion, which is dual\footnote{Two posets $\mathcal{P}_1$ and $\mathcal{P}_2$ are \emph{dual} to each other provided there exists an order-reversing isomorphism between $\mathcal{P}_1$ and $\mathcal{P}_2$.} to the face poset of $P$:
$$
\sigma(P) : = \{ T \subset S \mid  T = \sigma(f) \textrm{ for some face }  f \textrm{ of } P \},
$$
where $\sigma(f) := \{ s \in S \mid f \subset s \}$. For any subset $T \subset S$, we denote by $A_T$ the restriction of the Cartan matrix $A_P$ of $P$ to $T \times T$.

\begin{theorem}{\cite[Theorems 4 and 7]{V}}\label{thm:vinberg}
Let $P$ be a mirror $n$-polytope with facets $\{s\}_{s \in S}$ and with irreducible Cartan matrix $A_P$ of negative type. Let $T$ be a proper subset of $S$ (i.e. $T \neq \varnothing, S$). Then: 
\begin{enumerate}
\item\label{thm:vinberg_faces} If $A_T$ is of positive type and $\sharp T = k$, then $T \in \sigma(P)$ and its corresponding face $\cap_{s \in T} s $ is of dimension $n-k$.
\item\label{thm:vinberg_faces_zero} If $A_T$ is of zero type and of rank $n-1$, then $T \in \sigma(P)$ and the face $\cap_{s \in T} s $ is of dimension $0$, i.e. a vertex of $P$.
\end{enumerate}
\end{theorem}

\begin{remark}
Theorem \ref{thm:vinberg} is not explicitly stated in \cite[Theorems 4 and 7]{V}, but it is obtained by applying \cite[Theorem 4]{V} as in the proof of \cite[Theorem 7]{V}.
\end{remark}

\begin{remark}\label{rem:ridge}
Theorem \ref{thm:vinberg}.(\ref{thm:vinberg_faces}) tells us that for any mirror polytope $P$ with facets $\{s\}_{s\in S}$ and reflections $\{ \sigma_s = \mathrm{Id} - \alpha_s \otimes b_s \}_{s \in S}$, if $\alpha_s(b_t) \alpha_t(b_s) < 4$, then the intersection $s \cap t$ is a face of codimension $2$, i.e. a ridge of $P$.
\end{remark}

\subsection{Coxeter groups}\label{subsec:Coxeter_groups}

A \emph{Coxeter matrix} $M=(M_{st})_{s,t \in S}$ on a finite set $S$ is a symmetric matrix with the entries $M_{st} \in \{1,2, \dotsc, m, \dotsc,\infty \}$ such that the diagonal entries $M_{ss}=1$ and the others $M_{st} \neq 1$. To any Coxeter matrix $M=(M_{st})_{s,t \in S}$ is associated a \emph{Coxeter group} $W_{S,M}$ given by a presentation $\langle\, S \mid (st)^{M_{st}}=1 \textrm{ for } M_{st} \neq \infty \,\rangle$. We denote the Coxeter group $W_{S,M}$ also simply by $W,W_S$ or $W_M$. The \emph{rank} of $W_{S}$ is the cardinality $\sharp S$ of $S$.

\medskip

The \emph{Coxeter diagram} of $W_{S,M}$ is a labelled graph $\mathcal{G}_W$ such that (\emph{i}) the set of nodes (i.e. vertices) of $\mathcal{G}_W$ is the set $S$, (\emph{ii}) two nodes $s,t \in S$ are connected by an edge $\overline{st}$ of $\mathcal{G}_W$ if and only if $M_{st} \in \{3,\dotsc, m, \dotsc,\infty \}$, and (\emph{iii}) the edge $\overline{st}$ is labelled by $M_{st}$ if and only if $M_{st} > 3$. A Coxeter group $W$ is \emph{irreducible} if the Coxeter diagram $\mathcal{G}_W$ is connected.

\medskip

An irreducible Coxeter group $W$ is \emph{spherical} (resp. \emph{affine}) if it is finite (resp. infinite and virtually abelian). For a Coxeter group $W$ (not necessarily irreducible), each connected component of the Coxeter diagram $\mathcal{G}_W$ corresponds to a Coxeter group, called a \emph{component} of $W$. A Coxeter group $W$ is \emph{spherical} (resp. \emph{affine}) if each component of $W$ is spherical (resp. affine). We sometimes refer to Appendix \ref{classi_diagram} for the list of all the irreducible spherical and irreducible affine Coxeter diagrams.

\medskip

For each $T\subset S$, the subgroup $W'$ of $W$ generated by $T$ is called a \emph{standard subgroup of $W$}. It is well-known that $W'$ identifies with the Coxeter group $W_{T,M_{T}}$, where $M_{T}$ is the restriction of $M$ to $T \times T$. A subset $T \subset S$ is said to be “\emph{something}” if the Coxeter group $W_T$ is “something”. For example, the word “something” can be replaced by “spherical”, “affine” and so on. Two subsets $T, U \subset S$ are \emph{orthogonal} if $M_{tu}=2$ for every $t \in T$ and every $u \in U$. This relationship is denoted by $T \perp U$.

\subsection{Coxeter polytopes}

Recall that a \emph{Coxeter polytope} is a mirror polytope whose dihedral angles are submultiples of $\pi$, i.e. each dihedral angle is $\nicefrac{\pi}{m}$ for some integer $m \geqslant 2$ or $m=\infty$.

Let $P$ be a Coxeter polytope with the set of facets $S$ and the set of reflections $\{ \sigma_s = \mathrm{Id} - \alpha_s \otimes b_s \}_{s\in S}$. The \emph{Coxeter group $W_P$ of $P$} is the Coxeter group $W_{S,M}$ associated to a Coxeter matrix $M = (M_{st})_{s,t \in S}$ satisfying that $M_{st}=m_{st}$ if $\alpha_s(b_t) \alpha_t(b_s) = 4 \cos^2 (\nicefrac{\pi}{m_{st}})$ and $M_{st} = \infty$ if $\alpha_s(b_t) \alpha_t(b_s) \geq 4$. For a proper face $f$ of $P$ (i.e. $f \neq \varnothing$,  $P$), we write $\sigma (f) = \{ s \in S \mid f \subset s \}$ and $W_f := W_{\sigma(f)}$. 

\begin{theorem}[{Tits \cite[Chapter V]{MR0240238} for Tits simplex, and Vinberg \cite{V}}]\label{theo_vinberg}
$\,$\\Let $P$  be a Coxeter  polytope of $\Sb^n$ with Coxeter group $W_P$ and let $\G_P$ be the subgroup of $\mathrm{SL}^{\pm}_{n+1}\R$ generated by the reflections $\{ \sigma_s\}_{s \in S}$. Then:
\begin{enumerate}
\item The homomorphism $\sigma\colon W_P \rightarrow \Gamma_P \subset \mathrm{SL}^{\pm}_{n+1}\R$ defined by $\sigma(s) =
\sigma_s$ is an isomorphism.

\item The $\Gamma_P$-orbit of $P$ is a convex subset $\Cc_P$ of $\Sb^n$, and $\gamma \,\mathrm{Int}(P) \cap \mathrm{Int}(P) = \varnothing $ for all non-trivial $\gamma \in \Gamma_P$.

\item The group $\Gamma_P$ acts properly discontinuously on the interior $\O_P$ of $\Cc_P$.

\item An open proper face $f$ of $P$ lies in $\O_P$ if and only if the Coxeter group $W_f$ is spherical.
\end{enumerate}
\end{theorem}

Theorem \ref{theo_vinberg} tells us that $\O_P$ is a convex domain of $\Sb^n$, and that if $\Omega_P$ is properly convex then $\O_P/_{\Gamma_P}$ is a convex projective Coxeter orbifold.

\subsection{Relative hyperbolicity}\label{sec:rel_hyp}

Let $Y$  be a proper Gromov-hyperbolic space (see e.g. \cite[Section 2]{hruska} for a quick review and \cite[Part III.3]{bridson_haefliger} for  details on Gromov-hyperbolic spaces). We recall that for every isometry $\g$ of $Y$, exactly one of the following holds:
\begin{enumerate}
	\item 
	$\g$ fixes a point of $Y$.
	\item $\g$ fixes exactly one point of the Gromov boundary $\partial Y$ of $Y$.
	\item $\g$ fixes two points of  $\partial Y$.
\end{enumerate}
We say that $\g$ is \emph{parabolic} (resp. \emph{hyperbolic}) if (2) (resp. (3)) holds. Let $\Gamma$ be a subgroup of isometries of $Y$ that acts properly discontinuously. A subgroup of $\Gamma$ is \emph{parabolic} if it is infinite and contains no hyperbolic element. A parabolic subgroup fixes a unique point of $\partial Y$, called a \emph{parabolic point}. The stabiliser of a parabolic point is a maximal parabolic subgroup.

\medskip

Relative hyperbolicity has many equivalent definitions, see e.g. \cite[Section 3]{hruska}. We recall one of them, named \emph{cusp uniform action}. A group $\Gamma$ is \emph{relatively hyperbolic with respect to a collection $\Pc$ of subgroups} if there exist a proper Gromov-hyperbolic metric space $Y$ and a properly discontinuous effective action of $\Gamma$ on $Y$ by isometry such that:
\begin{itemize}
	\item the collection $\Pc$ is a set of representatives of the conjugacy classes of maximal parabolic subgroups of $\G$,
	
	\item there exists a $\Gamma$-equivariant collection $\Hc$ of disjoint open horoballs\footnote{We refer to \cite[Section 2]{hruska} for the notion of  horoball in Gromov-hyperbolic space.} centred at the parabolic points of $\Gamma$,
	
	\item the action of $\Gamma$ on $Y \smallsetminus U$ is cocompact, where $U$ denotes the union of the horoballs in~$\Hc$.
\end{itemize} 

For example, the fundamental group of a cusped hyperbolic $n$-manifold (or $n$-orbifold) is relatively hyperbolic with respect to its cusp subgroups, which are virtually $\Z^{n-1}$ \cite{bowditch,farb}.

\medskip

For $k\geqslant 2$, any $\Z^k$-subgroup $\Lambda$ of a relatively hyperbolic group $\Gamma$ with respect to a collection $\Pc$ of subgroups must lie in a conjugate of a subgroup $P \in \Pc$. Indeed, the centraliser of a hyperbolic element in a discrete subgroup of isometries of $Y$ is virtually $\Z$. Thus $\Lambda$ must contain a parabolic isometry $\delta$ with a unique fixed point $p\in \partial Y$, and any other element $\gamma \in \Lambda$ also has to fix $p$ since it commutes with $\delta$. So $\Lambda$ lies in the stabiliser of $p$.

\medskip

In particular the fundamental group $\pi_1 X$ of Theorem \ref{thm:main}, which is relatively hyperbolic with respect to the collection $\Pc = \{\pi_1T_i,\,\pi_1T'_i\}_i$ of rank-$2$ abelian subgroups, is not relatively hyperbolic with respect to any proper sub-collection of $\Pc$.

\medskip

We end this section by giving a criterion to determine when a Coxeter group is relatively hyperbolic with respect to a collection of standard subgroups. We will use this criterion in Section \ref{sec:caprace}.

\begin{theorem}[Moussong \cite{moussong} and Caprace \cite{caprace_cox_rel-hyp,caprace_erratum}]\label{moussong_caprace}
Let $W_{S}$ be a Coxeter group, and let $\Tc$ be a collection of subsets of $S$. Then the group $W_S$ is relatively hyperbolic with respect to $\{ W_T \mid T \in \Tc \}$ if and only if the following hold:
\begin{enumerate}
\item\label{thm:Caprace1} For every irreducible affine subset $U \subset S$ of rank $\geqslant 3$, there exists $T \in \Tc$ such that $U \subset T$.
\item\label{thm:Caprace2} For every pair of irreducible non-spherical subsets $S_1, S_2$ of $S$ with $S_1 \perp S_2$, there exists $T \in \Tc$ such that $S_1 \cup S_2 \subset T$.
\item\label{thm:Caprace3} For every pair $T, T' \in \mathcal{T}$ with $T \neq T'$, the intersection $T \cap T'$ is spherical.
\item\label{thm:Caprace4} For every $T \in \Tc$ and every irreducible non-spherical subset $U \subset T$, we have $U^{\perp} \subset T$, where $U^{\perp} := \{ s \in S \mid s \perp U\}$.
\end{enumerate}
\end{theorem}

\subsection{Operation on a simplex} \label{sec:truncated_simplex}

We introduce here three uniform\footnote{A polytope $P$ of dimension $n \geq 3$ (in the Euclidean space) is \emph{uniform} if it is a vertex-transitive polytope with uniform facets. A \emph{uniform polygon} is a regular polygon. By \emph{vertex-transitive}, we mean that the symmetry group of $P$ acts transitively on the set of vertices of $P$.} Euclidean 4-polytopes via truncation, rectification and bitruncation of the 4-simplex.

\medskip

Roughly speaking, by truncation, rectification and bitruncation of a regular polytope $P \subset \R^n$ we mean cutting uniformly $P$ at \emph{every} vertex with a hyperplane orthogonal to the line joining the vertex to the barycentre. This operation is nicely described in the classical book of Coxeter \cite[Section 8.1]{Cox}. Combinatorially, by collapsing some ridges of the bitruncated $P$ to vertices, one gets the rectified $P$. For example, a truncated (resp. rectified, resp. bitruncated) 3-simplex is a truncated tetrahedron (resp. an octahedron, resp. a truncated tetrahedron) in Figure \ref{fig:truncated_3_simplices}.

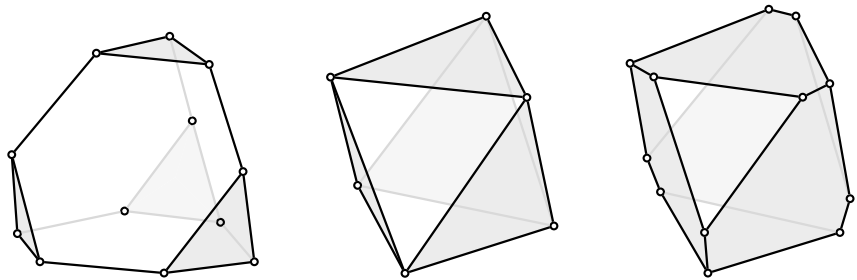
\begin{figure}[h!]
\begin{center}
\begin{tikzpicture}[thick,scale=0.75, every node/.style={transform shape}]
\fill[gray!7] (2.4,1.5) -- (1.5,0.9) -- (3.2,0.7) -- cycle;
\fill[gray!7] (2.4,1.5) -- (1.5,0.9) -- (2.7,2.5) -- cycle;
\fill[gray!7] (2.4,1.5) -- (3.2,0.7) -- (2.7,2.5) -- cycle;

\fill[gray!15] (0,0) -- (-0.5,1.9) -- (-0.4,0.5) -- cycle;
\fill[gray!15] (3.0,3.5) -- (2.3,4.0) -- (1.0,3.7) -- cycle;
\fill[gray!15] (2.2,-0.2) -- (3.8,0.0) -- (3.6,1.6) -- cycle;

\node[draw, fill=white, circle, inner sep=1pt, minimum size=1.2mm] (18) at (1.5,0.9)  {};
\node[draw, fill=gray!20, circle, inner sep=1pt, minimum size=1.2mm] (19) at (3.2,0.7)  {};
\node[draw, fill=white, circle, inner sep=1pt, minimum size=1.2mm] (20) at (2.7,2.5)  {};

\draw[line width=0.3mm, color=gray!30] (19) -- (20);
\draw[line width=0.3mm, color=gray!30] (18) -- (19);
\draw[line width=0.3mm, color=gray!30] (18) -- (20);

\node[draw, fill=white, circle, inner sep=1pt, minimum size=1.2mm] (1) at (0.0,0.0)  {};
\node[draw, fill=white, circle, inner sep=1pt, minimum size=1.2mm] (2) at (2.2,-0.2) {};
\node[draw, fill=white, circle, inner sep=1pt, minimum size=1.2mm] (3) at (3.8,0.0)  {};
\node[draw, fill=white, circle, inner sep=1pt, minimum size=1.2mm] (4) at (3.6,1.6)  {};
\node[draw, fill=white, circle, inner sep=1pt, minimum size=1.2mm] (5) at (3.0,3.5)  {};
\node[draw, fill=white, circle, inner sep=1pt, minimum size=1.2mm] (6) at (2.3,4.0)  {};
\node[draw, fill=white, circle, inner sep=1pt, minimum size=1.2mm] (7) at (1.0,3.7)  {};
\node[draw, fill=white, circle, inner sep=1pt, minimum size=1.2mm] (8) at (-0.5,1.9)  {};
\node[draw, fill=white, circle, inner sep=1pt, minimum size=1.2mm] (9) at (-0.4,0.5)  {};

\draw[line width=0.3mm, color=gray!30] (9) -- (18);
\draw[line width=0.3mm, color=gray!30] (3) -- (19);

\draw[line width=0.3mm, color=gray!30] (6) -- (20);

\draw[line width=0.3mm] (1) -- (2);
\draw[line width=0.3mm] (1) -- (8);
\draw[line width=0.3mm] (1) -- (9);
\draw[line width=0.3mm] (2) -- (3);
\draw[line width=0.3mm] (3) -- (4);
\draw[line width=0.3mm] (2) -- (4);
\draw[line width=0.3mm] (4) -- (5);
\draw[line width=0.3mm] (5) -- (6);
\draw[line width=0.3mm] (5) -- (7);
\draw[line width=0.3mm] (6) -- (7);
\draw[line width=0.3mm] (7) -- (8);
\draw[line width=0.3mm] (8) -- (9);
\end{tikzpicture}
\quad\quad
\begin{tikzpicture}[thick,scale=0.9, every node/.style={transform shape}]
\fill[gray!7] (0.6,1.2) -- (3.5,0.6) -- (2.5,3.7) -- cycle;
\fill[gray!15] (0.2,2.8) -- (2.5,3.7) -- (3.1,2.5) --  cycle;
\fill[gray!15] (0.2,2.8) -- (0.6,1.2) -- (1.3,-0.1) -- cycle;
\fill[gray!15] (3.5,0.6) -- (1.3,-0.1) -- (3.1,2.5) -- cycle;

\node[draw, fill=white, circle, inner sep=1pt, minimum size=1.0mm] (18) at (0.6,1.2)  {};
\node[draw, fill=white, circle, inner sep=1pt, minimum size=1.0mm] (19) at (3.5,0.6)  {};
\node[draw, fill=white, circle, inner sep=1pt, minimum size=1.0mm] (20) at (2.5,3.7)  {};

\draw[line width=0.3mm, color=gray!30] (19) -- (20);
\draw[line width=0.3mm, color=gray!30] (18) -- (19);
\draw[line width=0.3mm, color=gray!30] (18) -- (20);

\node[draw, fill=white, circle, inner sep=1pt, minimum size=1.0mm] (1) at (1.3,-0.1)  {};
\node[draw, fill=white, circle, inner sep=1pt, minimum size=1.0mm] (4) at (3.1,2.5)  {};
\node[draw, fill=white, circle, inner sep=1pt, minimum size=1.0mm] (7) at (0.2,2.8)  {};

\draw[line width=0.3mm] (1) -- (7);
\draw[line width=0.3mm] (1) -- (18);
\draw[line width=0.3mm] (1) -- (19);
\draw[line width=0.3mm] (19) -- (4);
\draw[line width=0.3mm] (1) -- (4);
\draw[line width=0.3mm] (4) -- (20);
\draw[line width=0.3mm] (4) -- (7);
\draw[line width=0.3mm] (20) -- (7);
\draw[line width=0.3mm] (7) -- (18);
\end{tikzpicture}
\quad\quad
\begin{tikzpicture}[thick,scale=0.9, every node/.style={transform shape}]
\fill[gray!7] (0.5,1.5) -- (0.7,1.0) -- (3.35,0.4) -- (3.5,0.9) -- (2.7,3.6) -- (2.3,3.7) -- cycle;
\fill[gray!15] (2.3,3.7) -- (2.7,3.6) -- (3.2,2.6) -- (2.8,2.4) -- (0.6,2.7) -- (0.25,2.9) -- cycle;
\fill[gray!15] (0.25,2.9) -- (0.5,1.5) -- (0.7,1.0) -- (1.4,-0.2) -- (1.35,0.4) -- (0.6,2.7) -- (0.25,2.9) -- cycle;
\fill[gray!15] (3.5,0.9) -- (3.35,0.4) -- (1.4,-0.2) -- (1.35,0.4) -- (2.8,2.4) -- (3.2,2.6) -- cycle;

\node[draw, fill=white, circle, inner sep=1pt, minimum size=1.0mm] (18_2) at (0.7,1.0)  {};
\node[draw, fill=white, circle, inner sep=1pt, minimum size=1.0mm] (18_3) at (0.5,1.5)  {};
\node[draw, fill=white, circle, inner sep=1pt, minimum size=1.0mm] (19_1) at (3.35,0.4)  {};
\node[draw, fill=white, circle, inner sep=1pt, minimum size=1.0mm] (19_3) at (3.5,0.9)  {};
\node[draw, fill=white, circle, inner sep=1pt, minimum size=1.0mm] (20_2) at (2.3,3.7)  {};
\node[draw, fill=white, circle, inner sep=1pt, minimum size=1.0mm] (20_3) at (2.7,3.6)  {};

\draw[line width=0.3mm] (20_2) -- (20_3);

\draw[line width=0.3mm, color=gray!30] (19_3) -- (20_3);
\draw[line width=0.3mm, color=gray!30] (18_2) -- (19_1);
\draw[line width=0.3mm, color=gray!30] (18_3) -- (20_2);

\draw[line width=0.3mm] (18_2) -- (18_3);

\draw[line width=0.3mm] (19_1) -- (19_3);

\node[draw, fill=white, circle, inner sep=1pt, minimum size=1.0mm] (1_1) at (1.4,-0.2)  {};
\node[draw, fill=white, circle, inner sep=1pt, minimum size=1.0mm] (1_3) at (1.35,0.4)  {};
\node[draw, fill=white, circle, inner sep=1pt, minimum size=1.0mm] (4_1) at (2.8,2.4)  {};
\node[draw, fill=white, circle, inner sep=1pt, minimum size=1.0mm] (4_2) at (3.2,2.6)  {};
\node[draw, fill=white, circle, inner sep=1pt, minimum size=1.0mm] (7_1) at (0.25,2.9)  {};
\node[draw, fill=white, circle, inner sep=1pt, minimum size=1.0mm] (7_3) at (0.6,2.7)  {};

\draw[line width=0.3mm] (4_1) -- (4_2);

\draw[line width=0.3mm] (1_1) -- (1_3);

\draw[line width=0.3mm] (7_1) -- (7_3);

\draw[line width=0.3mm] (1_3) -- (7_3);
\draw[line width=0.3mm] (1_1) -- (18_2);
\draw[line width=0.3mm] (1_1) -- (19_1);
\draw[line width=0.3mm] (19_3) -- (4_2);
\draw[line width=0.3mm] (1_3) -- (4_1);
\draw[line width=0.3mm] (4_2) -- (20_3);
\draw[line width=0.3mm] (4_1) -- (7_3);
\draw[line width=0.3mm] (20_2) -- (7_1);
\draw[line width=0.3mm] (7_1) -- (18_3);
\end{tikzpicture}
\end{center}
\caption{\footnotesize The truncated (left), rectified (middle), and bitruncated (right) 3-simplex.}
\label{fig:truncated_3_simplices}
\end{figure}

\medskip

We now explain in detail this operation for the 4-simplex. Consider a regular 4-simplex $\Delta\subset\R^4$ with barycentre the origin and vertices $v_1,\ldots,v_5$. We denote by $F_i$ the facet of $\Delta$ opposite to $v_i$ and by $H_i$ the closed half-space containing the origin with $F_i\subset\partial H_i$. Then $\Delta=H_1\cap\dotsc\cap H_5$.

\medskip

Let $c :=|v_i|$ and fix a positive parameter $s\leq c$. We denote by $H'_i$ the closed half-space (depending on $s$) containing the origin such that $\partial H'_i$ is orthogonal to $ v_i$ and $\frac{s}{c}\, v_i \in\partial H'_i$, and we set
$$Q_s= H_1\cap\ldots\cap H_5\cap H'_1\cap\ldots\cap H'_5.$$
Note that $Q_{c}=\Delta$ is the original simplex. There exist some numbers $0 < a < b < c$ such that the combinatorics of the 4-polytope $Q_s$ is constant for $s$ in $(a,b)$ and $(b,c)$, and changes at $s=a$, $b$ and $c$. The polytope $Q_s$ (depicted in Figure \ref{fig:truncated_simplices}) is called:  
\begin{itemize}
\item a \emph{truncated 4-simplex} for $s\in(b,c)$,
\item a \emph{rectified 4-simplex} for $s=b$,
\item a \emph{bitruncated 4-simplex} for $s\in(a,b)$.
\end{itemize}

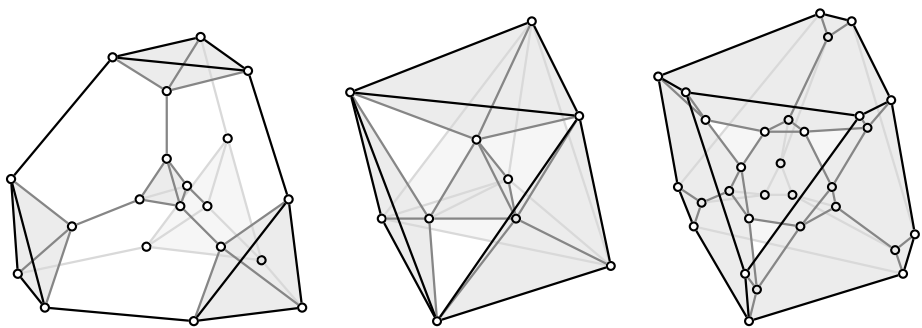
\begin{figure}[h!]
\begin{center}
\begin{tikzpicture}[thick,scale=0.9, every node/.style={transform shape}]
\fill[gray!7] (2.4,1.5) -- (1.5,0.9) -- (3.2,0.7) -- cycle;
\fill[gray!7] (2.4,1.5) -- (1.5,0.9) -- (2.7,2.5) -- cycle;
\fill[gray!7] (2.4,1.5) -- (3.2,0.7) -- (2.7,2.5) -- cycle;

\fill[gray!15] (0,0) -- (-0.5,1.9) -- (-0.4,0.5) -- cycle;
\fill[gray!15] (0,0) -- (-0.5,1.9) -- (0.4,1.2) -- cycle;
\fill[gray!15] (3.0,3.5) -- (2.3,4.0) -- (1.0,3.7) -- cycle;
\fill[gray!15] (3.0,3.5) -- (1.8,3.2) -- (1.0,3.7) -- cycle;
\fill[gray!15] (2.2,-0.2) -- (3.8,0.0) -- (3.6,1.6) -- cycle;
\fill[gray!15] (2.2,-0.2) -- (2.6,0.9) -- (3.6,1.6) -- cycle;

\fill[gray!15] (1.4,1.6) -- (2.0,1.5) -- (1.8,2.2) -- cycle;
\fill[gray!15] (2.1,1.8) -- (2.0,1.5) -- (1.8,2.2) -- cycle;

\node[draw, fill=white, circle, inner sep=1pt, minimum size=1.2mm] (17) at (2.4,1.5)  {};
\node[draw, fill=white, circle, inner sep=1pt, minimum size=1.2mm] (18) at (1.5,0.9)  {};
\node[draw, fill=gray!20, circle, inner sep=1pt, minimum size=1.2mm] (19) at (3.2,0.7)  {};
\node[draw, fill=white, circle, inner sep=1pt, minimum size=1.2mm] (20) at (2.7,2.5)  {};

\draw[line width=0.3mm, color=gray!30] (19) -- (20);
\draw[line width=0.3mm, color=gray!30] (18) -- (19);
\draw[line width=0.3mm, color=gray!30] (18) -- (20);
\draw[line width=0.3mm, color=gray!30] (17) -- (20);
\draw[line width=0.3mm, color=gray!30] (17) -- (19);
\draw[line width=0.3mm, color=gray!30] (17) -- (18);

\node[draw, fill=white, circle, inner sep=1pt, minimum size=1.2mm] (11) at (1.4,1.6)  {};
\node[draw, fill=white, circle, inner sep=1pt, minimum size=1.2mm] (12) at (2.0,1.5)  {};
\node[draw, fill=white, circle, inner sep=1pt, minimum size=1.2mm] (13) at (1.8,2.2)  {};
\node[draw, fill=white, circle, inner sep=1pt, minimum size=1.2mm] (14) at (2.1,1.8)  {};
\node[draw, fill=white, circle, inner sep=1pt, minimum size=1.2mm] (15) at (2.6,0.9)  {};
\node[draw, fill=white, circle, inner sep=1pt, minimum size=1.2mm] (16) at (1.8,3.2)  {};

\draw[line width=0.3mm, color=gray!95] (11) -- (12);
\draw[line width=0.3mm, color=gray!95] (11) -- (13);
\draw[line width=0.3mm, color=gray!30] (11) -- (14);
\draw[line width=0.3mm, color=gray!95] (12) -- (13);
\draw[line width=0.3mm, color=gray!95] (13) -- (14);
\draw[line width=0.3mm, color=gray!95] (12) -- (14);
\draw[line width=0.3mm, color=gray!95] (12) -- (15);

\node[draw, fill=white, circle, inner sep=1pt, minimum size=1.2mm] (1) at (0.0,0.0)  {};
\node[draw, fill=white, circle, inner sep=1pt, minimum size=1.2mm] (2) at (2.2,-0.2) {};
\node[draw, fill=white, circle, inner sep=1pt, minimum size=1.2mm] (3) at (3.8,0.0)  {};
\node[draw, fill=white, circle, inner sep=1pt, minimum size=1.2mm] (4) at (3.6,1.6)  {};
\node[draw, fill=white, circle, inner sep=1pt, minimum size=1.2mm] (5) at (3.0,3.5)  {};
\node[draw, fill=white, circle, inner sep=1pt, minimum size=1.2mm] (6) at (2.3,4.0)  {};
\node[draw, fill=white, circle, inner sep=1pt, minimum size=1.2mm] (7) at (1.0,3.7)  {};
\node[draw, fill=white, circle, inner sep=1pt, minimum size=1.2mm] (8) at (-0.5,1.9)  {};
\node[draw, fill=white, circle, inner sep=1pt, minimum size=1.2mm] (9) at (-0.4,0.5)  {};
\node[draw, fill=white, circle, inner sep=1pt, minimum size=1.2mm] (10) at (0.4,1.2)  {};

\draw[line width=0.3mm, color=gray!30] (9) -- (18);
\draw[line width=0.3mm, color=gray!30] (3) -- (19);
\draw[line width=0.3mm, color=gray!95] (3) -- (15);

\draw[line width=0.3mm, color=gray!30] (6) -- (20);

\draw[line width=0.3mm, color=gray!95] (1) -- (10);
\draw[line width=0.3mm, color=gray!95] (8) -- (10);
\draw[line width=0.3mm, color=gray!95] (9) -- (10);

\draw[line width=0.3mm, color=gray!95] (2) -- (15);
\draw[line width=0.3mm, color=gray!95] (4) -- (15);
\draw[line width=0.3mm, color=gray!95] (10) -- (11);
\draw[line width=0.3mm, color=gray!95] (13) -- (16);

\draw[line width=0.3mm, color=gray!95] (14) -- (17);

\draw[line width=0.3mm, color=gray!95] (6) -- (16);
\draw[line width=0.3mm, color=gray!95] (7) -- (16);
\draw[line width=0.3mm, color=gray!95] (5) -- (16);
\draw[line width=0.3mm] (1) -- (2);
\draw[line width=0.3mm] (1) -- (8);
\draw[line width=0.3mm] (1) -- (9);
\draw[line width=0.3mm] (2) -- (3);
\draw[line width=0.3mm] (3) -- (4);
\draw[line width=0.3mm] (2) -- (4);
\draw[line width=0.3mm] (4) -- (5);
\draw[line width=0.3mm] (5) -- (6);
\draw[line width=0.3mm] (5) -- (7);
\draw[line width=0.3mm] (6) -- (7);
\draw[line width=0.3mm] (7) -- (8);
\draw[line width=0.3mm] (8) -- (9);
\end{tikzpicture}
\quad
\begin{tikzpicture}[thick,scale=1.05, every node/.style={transform shape}]
\fill[gray!7] (0.6,1.2) -- (3.5,0.6) -- (2.5,3.7) -- cycle;
\fill[gray!15] (0.2,2.8) -- (2.5,3.7) -- (3.1,2.5) -- (1.8,2.2) -- cycle;
\fill[gray!15] (0.2,2.8) -- (0.6,1.2) -- (1.3,-0.1) -- (1.2,1.2) -- cycle;
\fill[gray!15] (2.2,1.7) -- (2.3,1.2) -- (1.2,1.2) -- (1.8,2.2) -- cycle;
\fill[gray!15] (3.5,0.6) -- (1.3,-0.1) -- (3.1,2.5) -- cycle;

\node[draw, fill=white, circle, inner sep=1pt, minimum size=1.0mm] (17) at (2.2,1.7)  {};
\node[draw, fill=white, circle, inner sep=1pt, minimum size=1.0mm] (18) at (0.6,1.2)  {};
\node[draw, fill=white, circle, inner sep=1pt, minimum size=1.0mm] (19) at (3.5,0.6)  {};
\node[draw, fill=white, circle, inner sep=1pt, minimum size=1.0mm] (20) at (2.5,3.7)  {};

\draw[line width=0.3mm, color=gray!30] (19) -- (20);
\draw[line width=0.3mm, color=gray!30] (18) -- (19);
\draw[line width=0.3mm, color=gray!30] (18) -- (20);
\draw[line width=0.3mm, color=gray!30] (17) -- (20);
\draw[line width=0.3mm, color=gray!30] (17) -- (19);
\draw[line width=0.3mm, color=gray!30] (17) -- (18);

\node[draw, fill=white, circle, inner sep=1pt, minimum size=1.0mm] (10) at (1.2,1.2)  {};
\node[draw, fill=gray!20, circle, inner sep=1pt, minimum size=1.0mm] (12) at (2.3,1.2)  {};
\node[draw, fill=white, circle, inner sep=1pt, minimum size=1.0mm] (13) at (1.8,2.2)  {};

\draw[line width=0.3mm, color=gray!95] (10) -- (12);
\draw[line width=0.3mm, color=gray!95] (10) -- (13);
\draw[line width=0.3mm, color=gray!30] (10) -- (17);
\draw[line width=0.3mm, color=gray!95] (12) -- (13);
\draw[line width=0.3mm, color=gray!95] (13) -- (17);
\draw[line width=0.3mm, color=gray!95] (12) -- (17);

\node[draw, fill=white, circle, inner sep=1pt, minimum size=1.0mm] (1) at (1.3,-0.1)  {};
\node[draw, fill=white, circle, inner sep=1pt, minimum size=1.0mm] (4) at (3.1,2.5)  {};
\node[draw, fill=white, circle, inner sep=1pt, minimum size=1.0mm] (7) at (0.2,2.8)  {};

\draw[line width=0.3mm, color=gray!95] (19) -- (12);

\draw[line width=0.3mm, color=gray!95] (1) -- (10);
\draw[line width=0.3mm, color=gray!95] (7) -- (10);
\draw[line width=0.3mm, color=gray!95] (18) -- (10);

\draw[line width=0.3mm, color=gray!95] (1) -- (12);
\draw[line width=0.3mm, color=gray!95] (4) -- (12);

\draw[line width=0.3mm, color=gray!95] (20) -- (13);
\draw[line width=0.3mm, color=gray!95] (7) -- (13);
\draw[line width=0.3mm, color=gray!95] (4) -- (13);
\draw[line width=0.3mm] (1) -- (7);
\draw[line width=0.3mm] (1) -- (18);
\draw[line width=0.3mm] (1) -- (19);
\draw[line width=0.3mm] (19) -- (4);
\draw[line width=0.3mm] (1) -- (4);
\draw[line width=0.3mm] (4) -- (20);
\draw[line width=0.3mm] (4) -- (7);
\draw[line width=0.3mm] (20) -- (7);
\draw[line width=0.3mm] (7) -- (18);
\end{tikzpicture}
\quad
\begin{tikzpicture}[thick,scale=1.05, every node/.style={transform shape}]
\fill[gray!7] (0.5,1.5) -- (0.7,1.0) -- (3.35,0.4) -- (3.5,0.9) -- (2.7,3.6) -- (2.3,3.7) -- cycle;
\fill[gray!15] (2.3,3.7) -- (2.7,3.6) -- (3.2,2.6) -- (2.9,2.25) -- (2.1,2.2) -- (1.6,2.2) -- (0.85,2.35) -- (0.25,2.9) -- cycle;
\fill[gray!15] (0.25,2.9) -- (0.5,1.5) -- (0.7,1.0) -- (1.4,-0.2) -- (1.5,0.2) -- (1.4,1.1) -- (1.3,1.75) -- (0.85,2.35) -- cycle;
\fill[gray!15] (2.45,1.5) -- (2.05,1.0) -- (1.4,1.1) -- (1.3,1.75) -- (1.6,2.2) -- (2.1,2.2) -- cycle;
\fill[gray!15] (3.5,0.9) -- (3.35,0.4) -- (1.4,-0.2) -- (1.35,0.4) -- (2.8,2.4) -- (3.2,2.6) -- cycle;

\node[draw, fill=gray!15, circle, inner sep=1pt, minimum size=1.0mm] (17_1) at (1.8,1.8)  {};
\node[draw, fill=gray!15, circle, inner sep=1pt, minimum size=1.0mm] (17_2) at (1.95,1.4)  {};
\node[draw, fill=gray!15, circle, inner sep=1pt, minimum size=1.0mm] (17_3) at (1.6,1.4)  {};
\node[draw, fill=gray!15, circle, inner sep=1pt, minimum size=1.0mm] (18_1) at (0.8,1.3)  {};
\node[draw, fill=white, circle, inner sep=1pt, minimum size=1.0mm] (18_2) at (0.7,1.0)  {};
\node[draw, fill=white, circle, inner sep=1pt, minimum size=1.0mm] (18_3) at (0.5,1.5)  {};
\node[draw, fill=white, circle, inner sep=1pt, minimum size=1.0mm] (19_1) at (3.35,0.4)  {};
\node[draw, fill=gray!15, circle, inner sep=1pt, minimum size=1.0mm] (19_2) at (3.25,0.7)  {};
\node[draw, fill=white, circle, inner sep=1pt, minimum size=1.0mm] (19_3) at (3.5,0.9)  {};
\node[draw, fill=gray!15, circle, inner sep=1pt, minimum size=1.0mm] (20_1) at (2.4,3.4)  {};
\node[draw, fill=white, circle, inner sep=1pt, minimum size=1.0mm] (20_2) at (2.3,3.7)  {};
\node[draw, fill=white, circle, inner sep=1pt, minimum size=1.0mm] (20_3) at (2.7,3.6)  {};

\draw[line width=0.3mm] (20_2) -- (20_3);
\draw[line width=0.3mm, color=gray!95] (20_1) -- (20_2);
\draw[line width=0.3mm, color=gray!95] (20_1) -- (20_3);

\draw[line width=0.3mm, color=gray!30] (19_3) -- (20_3);
\draw[line width=0.3mm, color=gray!30] (18_2) -- (19_1);
\draw[line width=0.3mm, color=gray!30] (18_3) -- (20_2);
\draw[line width=0.3mm, color=gray!30] (17_1) -- (20_1);
\draw[line width=0.3mm, color=gray!30] (17_2) -- (19_2);
\draw[line width=0.3mm, color=gray!30] (17_3) -- (18_1);

\node[draw, fill=white, circle, inner sep=1pt, minimum size=1.0mm] (10_1) at (1.4,1.1)  {};
\node[draw, fill=gray!15, circle, inner sep=1pt, minimum size=1.0mm] (10_2) at (1.15,1.45)  {};
\node[draw, fill=white, circle, inner sep=1pt, minimum size=1.0mm] (10_3) at (1.3,1.75)  {};
\node[draw, fill=gray!15, circle, inner sep=1pt, minimum size=1.0mm] (12_1) at (2.05,1.0)  {};
\node[draw, fill=gray!15, circle, inner sep=1pt, minimum size=1.0mm] (12_2) at (2.45,1.5)  {};
\node[draw, fill=gray!15, circle, inner sep=1pt, minimum size=1.0mm] (12_3) at (2.5,1.25)  {};
\node[draw, fill=white, circle, inner sep=1pt, minimum size=1.0mm] (13_1) at (1.6,2.2)  {};
\node[draw, fill=gray!15, circle, inner sep=1pt, minimum size=1.0mm] (13_2) at (1.9,2.35)  {};
\node[draw, fill=white, circle, inner sep=1pt, minimum size=1.0mm] (13_3) at (2.1,2.2)  {};

\draw[line width=0.3mm, color=gray!30] (17_2) -- (17_3);
\draw[line width=0.3mm, color=gray!30] (17_1) -- (17_2);
\draw[line width=0.3mm, color=gray!30] (17_1) -- (17_3);

\draw[line width=0.3mm, color=gray!95] (19_2) -- (12_3);
\draw[line width=0.3mm, color=gray!95] (18_1) -- (10_2);

\draw[line width=0.3mm, color=gray!95] (10_2) -- (10_3);
\draw[line width=0.3mm, color=gray!95] (10_1) -- (10_2);

\draw[line width=0.3mm, color=gray!30] (10_2) -- (17_3);
\draw[line width=0.3mm, color=gray!30] (13_2) -- (17_1);
\draw[line width=0.3mm, color=gray!30] (12_3) -- (17_2);

\draw[line width=0.3mm, color=gray!95] (10_1) -- (12_1);
\draw[line width=0.3mm, color=gray!95] (10_3) -- (13_1);
\draw[line width=0.3mm, color=gray!95] (12_2) -- (13_3);

\draw[line width=0.3mm, color=gray!95] (13_2) -- (13_3);
\draw[line width=0.3mm, color=gray!95] (13_1) -- (13_2);
\draw[line width=0.3mm, color=gray!95] (13_1) -- (13_3);

\draw[line width=0.3mm, color=gray!95] (12_2) -- (12_3);
\draw[line width=0.3mm, color=gray!95] (12_1) -- (12_2);
\draw[line width=0.3mm, color=gray!95] (12_1) -- (12_3);

\draw[line width=0.3mm] (18_2) -- (18_3);
\draw[line width=0.3mm, color=gray!95] (18_1) -- (18_2);
\draw[line width=0.3mm, color=gray!95] (18_1) -- (18_3);

\draw[line width=0.3mm, color=gray!95] (19_2) -- (19_3);
\draw[line width=0.3mm, color=gray!95] (19_1) -- (19_2);
\draw[line width=0.3mm] (19_1) -- (19_3);

\node[draw, fill=gray!15, circle, inner sep=1pt, minimum size=1.0mm] (1_2) at (1.5,0.2)  {};
\draw[line width=0.3mm, color=gray!95] (1_2) -- (10_1);
\node[draw, fill=white, circle, inner sep=1pt, minimum size=1.0mm] (1_1) at (1.4,-0.2)  {};
\node[draw, fill=white, circle, inner sep=1pt, minimum size=1.0mm] (1_3) at (1.35,0.4)  {};
\node[draw, fill=white, circle, inner sep=1pt, minimum size=1.0mm] (4_1) at (2.8,2.4)  {};
\node[draw, fill=white, circle, inner sep=1pt, minimum size=1.0mm] (4_2) at (3.2,2.6)  {};
\node[draw, fill=gray!15, circle, inner sep=1pt, minimum size=1.0mm] (4_3) at (2.9,2.25)  {};
\node[draw, fill=white, circle, inner sep=1pt, minimum size=1.0mm] (7_1) at (0.25,2.9)  {};
\node[draw, fill=white, circle, inner sep=1pt, minimum size=1.0mm] (7_2) at (0.85,2.35)  {};
\draw[line width=0.3mm, color=gray!95] (7_2) -- (13_1);
\draw[line width=0.3mm, color=gray!95] (7_2) -- (10_3);

\node[draw, fill=white, circle, inner sep=1pt, minimum size=1.0mm] (7_3) at (0.6,2.7)  {};

\draw[line width=0.3mm, color=gray!95] (4_2) -- (4_3);
\draw[line width=0.3mm] (4_1) -- (4_2);
\draw[line width=0.3mm, color=gray!95] (4_1) -- (4_3);

\draw[line width=0.3mm, color=gray!95] (1_2) -- (1_3);
\draw[line width=0.3mm, color=gray!95] (1_1) -- (1_2);
\draw[line width=0.3mm] (1_1) -- (1_3);

\draw[line width=0.3mm, color=gray!95] (7_2) -- (7_3);
\draw[line width=0.3mm, color=gray!95] (7_1) -- (7_2);
\draw[line width=0.3mm] (7_1) -- (7_3);

\draw[line width=0.3mm, color=gray!95] (10_1) -- (10_3);

\draw[line width=0.3mm, color=gray!95] (1_2) -- (12_1);
\draw[line width=0.3mm, color=gray!95] (4_3) -- (12_2);

\draw[line width=0.3mm, color=gray!95] (20_1) -- (13_2);
\draw[line width=0.3mm, color=gray!95] (4_3) -- (13_3);
\draw[line width=0.3mm] (1_3) -- (7_3);
\draw[line width=0.3mm] (1_1) -- (18_2);
\draw[line width=0.3mm] (1_1) -- (19_1);
\draw[line width=0.3mm] (19_3) -- (4_2);
\draw[line width=0.3mm] (1_3) -- (4_1);
\draw[line width=0.3mm] (4_2) -- (20_3);
\draw[line width=0.3mm] (4_1) -- (7_3);
\draw[line width=0.3mm] (20_2) -- (7_1);
\draw[line width=0.3mm] (7_1) -- (18_3);
\end{tikzpicture}
\end{center}
\caption{\footnotesize The Schlegel diagrams of the truncated (left), rectified (middle), and bitruncated (right) 4-simplex. The facets $F'_i$ are coloured darker than $F_i$ (cf. Table \ref{table:comb} below).}
\label{fig:truncated_simplices}
\end{figure}

\medskip

The rectified simplex $Q_{b}$ is in fact the convex hull of the midpoints of the edges of the regular simplex $\Delta=Q_{c}$, while $Q_a$ is another rectified simplex. For all $s\in[a,c]$, the polytope $Q_s$ is uniform.

\subsection{The combinatorics of the rectified and bitruncated 4-simplices}\label{sec:combi}

We describe the combinatorics of $Q_s$ for $s \in (a,b]$.

\medskip

The link of a vertex of the rectified (resp. bitruncated) 4-simplex is a triangular prism (resp. a tetrahedron) as in Figure \ref{fig:link_comb}. Each vertex of $Q_s$ is the intersection of facets $F'_{i_1}\cap F'_{i_2}\cap F_{i_3}\cap F_{i_4}\cap F_{i_5}$ for all distinct $i_1,\ldots,i_5$ (resp. $F'_{i_1}\cap F'_{i_2}\cap F_{i_3}\cap F_{i_4}$ for all distinct $i_1,\ldots,i_4$), where $F_i$ and $F'_i$ denote the facets of $Q_s$ whose supporting hyperplanes are $\partial H_i$ and $\partial H'_i$, respectively. The 10 facets of $Q_s$ are divided into 5 octahedra (resp. truncated tetrahedra) $F_i\subset \partial H_i$, and 5 tetrahedra (resp. truncated tetrahedra) $F'_i\subset \partial H'_i$. For all $i\neq j$, the ridge $F_i\cap F_j$ is a triangle, $F'_i\cap F'_j$ is a vertex (resp. triangle), and $F_i\cap F'_j$ is a triangle (resp. hexagon), while $F_i\cap F'_i=\varnothing$ (see  Table \ref{table:comb}).

\begin{table}[!h]
\begin{center}
\begin{tabular}{c||c|c|c}
 			   	   & truncated 4-simplex    & rectified 4-simplex & bitruncated 4-simplex \\
 \hline
 $F_i$		  	   & truncated tetrahedron  & octahedron          &  truncated tetrahedron \\
 $F'_j$ 		   & tetrahedron            & tetrahedron         & truncated tetrahedron \\
 \hline
 $F_i \cap F_j$    & hexagon                & triangle            & triangle  \\
 $F_i \cap F'_i$   & $\varnothing$          & $\varnothing$       & $\varnothing$ \\
 $F_i \cap F'_j$   & triangle               & triangle            & hexagon   \\
 $F'_i \cap F'_j$  & $\varnothing$          & vertex              & triangle  \\
 \hline
\end{tabular}
\end{center}
\caption{\footnotesize Some information on the faces of the truncated, rectified and bitruncated $4$-simplex. The symbols $i,j$ are two distinct indices in $\{1,\ldots, 5 \}$.
}\label{table:comb}
\end{table}

\begin{figure}[!h]
\centering
\begin{tabular}{ccc}
\begin{tikzpicture}[line cap=round,line join=round,>=triangle 45,x=1cm,y=1cm]
\draw [line width=1.5pt] (-1,4)-- (0,3);
\draw [line width=1.5pt] (0,3)-- (1,4);
\draw [line width=1.5pt] (1,4)-- (-1,4);
\draw [line width=1.5pt] (-1,4)-- (-1,0.5);
\draw [line width=1.5pt] (-1,0.5)-- (1,0.5);
\draw [line width=1.5pt] (1,0.5)-- (0,1.5);
\draw [line width=1.5pt] (0,1.5)-- (-1,0.5);
\draw [line width=1.5pt] (1,0.5)-- (1,4);
\draw [line width=1.5pt] (0,3)-- (0,1.5);
\node[draw,circle, inner sep=0.2pt, minimum size=0.2pt] at (0,3.6) {$i'_1$};
\node[draw,circle, inner sep=0.2pt, minimum size=0.2pt] at (0,0.93)  {$i'_2$};
\node[draw,circle, inner sep=0.2pt, minimum size=0.2pt] at (0.43,2.4)  {$i_4$};
\node[draw,circle, inner sep=0.2pt, minimum size=0.2pt] at (-0.57,2.4)  {$i_3$};
\node[draw,circle, inner sep=0.2pt, minimum size=0.2pt] at (-1.58,2.4)  {$i_5$};
\end{tikzpicture}
& $\qquad \qquad$
\begin{tikzpicture}[line cap=round,line join=round,>=triangle 45,x=1cm,y=1cm]
\draw [line width=1.5pt] (0,1.8)-- (0,-0.2);
\draw [line width=1.5pt] (-1.7320508075688774,-1.2)-- (0,-0.2);
\draw [line width=1.5pt] (0,-0.2)-- (1.7320508075688774,-1.2);
\draw [line width=1.5pt] (1.7320508075688774,-1.2)-- (0,1.8);
\draw [line width=1.5pt] (0,1.8)-- (-1.7320508075688774,-1.2);
\draw [line width=1.5pt] (-1.7320508075688774,-1.2)-- (1.7320508075688774,-1.2);

\node[draw,circle, inner sep=0.2pt, minimum size=0.2pt] at (-0.45,0.22) {$i'_1$};
\node[draw,circle, inner sep=0.2pt, minimum size=0.2pt] at (0.45,0.22) {$i'_2$};
\node[draw,circle, inner sep=0.2pt, minimum size=0.2pt] at (0,-0.75) {$i_3$};
\node[draw,circle, inner sep=0.2pt, minimum size=0.2pt] at (0,-1.6) {$i_4$};
\end{tikzpicture}
\\
\end{tabular}
 \caption{\footnotesize The vertex link of the rectified (left) and bitruncated (right) 4-simplex $Q_s$. A facet of the link labelled by $i$ (resp $i'$) corresponds to the facet $F_i$ (resp. $F'_i$) of the 4-polytope $Q_s$.}
 \label{fig:link_comb}
\end{figure}
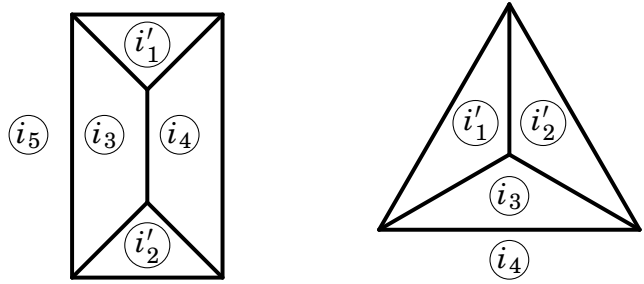

\medskip

The following is a simple observation, but it will be useful later to prove Theorem \ref{thm:main} (more precisely, Proposition \ref{prop:P_theta}). 

\begin{lemma}\label{lem:combi}
Let $Q_s$ be the rectified or bitruncated $4$-simplex for $s \in (a,b]$. Relabel each facet $F'_i$ of $Q_s$ with $F_{i'}$, and let $S := \{1',\ldots,5',1,\ldots, 5\}$.
\begin{enumerate}
\item\label{lem:combi_rectified} In the case that $Q_s$ is the rectified $4$-simplex, i.e. $s = b$, each vertex of $Q_s$ corresponds to a subset $\{i',j',k,l,m\} \subset S $ with $\sharp \{ i,j,k,l,m\} = 5$ and each edge of $Q_s$ corresponds to a subset $\{i',j,k\} \subset S $ with $\sharp \{ i,j,k \} = 3$.

\item\label{lem:combi_bitruncated} In the case that $Q_s$ is the bitruncated $4$-simplex, i.e. $a< s < b$, each vertex of $Q_s$ corresponds to a subset $\{i',j',k,l\} \subset S$ with $\sharp \{ i,j,k,l\} = 4$.
\end{enumerate}
\end{lemma}

\subsection{The ideal hyperbolic rectified 4-simplex}\label{subsec:hyperbolic_rectified}

Every vertex-transitive polytope $Q\subset\R^n$ can be realised as an ideal hyperbolic $n$-polytope, obtained by interpreting the ball in which $Q$ is inscribed as a projective model of the hyperbolic $n$-space $\Hb^n$. What is nice about the rectified simplex of dimension $n \leq 4$ is that its regular ideal hyperbolic realisation $R\subset\Hb^n$ is a Coxeter polytope.\footnote{It is well-known that the link of the regular ideal hyperbolic rectified $n$-simplex is a Euclidean right simplicial $(n-1)$-prism with regular $(n-2)$-simplicial bases, and the dihedral angle of the Euclidean regular $(n-2)$-simplex is $\arccos(\nicefrac{1}{(n-2)})$. The latter is $\nicefrac{\pi}{m}$ for some integer $m \geqslant 2$ if and only if $n \leqslant 4$. Hence, $R$ is a Coxeter polytope if and only if $n \leqslant 4$.} For instance, the polytope $R$ of dimension 3 is a right-angled hyperbolic octahedron. 

\medskip

Let $R\subset\Hb^4$ be the ideal rectified 4-simplex. The facets of $R$ are regular ideal Coxeter $3$-polytopes: five right-angled octahedra $F_i$ and five $\nicefrac\pi3$-angled tetrahedra $F'_i$. The horospherical link of any (ideal) vertex of $R$ is a Euclidean right triangular prism with equilateral bases (see the left of Figure \ref{fig:link_comb}). Thus, for all $i\neq j$ the dihedral angle at a ridge $F_i\cap F_j$ (resp. $F_i  \cap F'_j$) is $\nicefrac{\pi}{3}$ (resp. $\nicefrac{\pi}{2}$), while $F'_i$ and $F'_j$ are parallel, i.e. the facet $F'_i$ is tangent to $F'_j$ at infinity. 

\begin{remark}\label{rem:mathematica}
The bitruncated 4-simplex is ``combinatorially'' a ``filling'' of the ideal rectified 4-simplex in the sense that the latter is obtained from the former by collapsing each triangular ridge $F'_i\cap F'_j$ for $i \neq j$ to a point and removing it. We call such triangles the \emph{filling ridges} of the bitruncated 4-simplex.
\end{remark}

\subsection{Decomposability and Euler characteristics}

We conclude the section with a remark on convex projective manifolds that has probably been noticed by many experts of the subject, but whose explicit statement seems to miss in the literature. This remark is a simple consequence of works of Vey, Benoist and Gottlieb.

\begin{fact}\label{fact:indecomposable}
If a convex projective manifold $\Omega/_\Gamma$ is decomposable, then $\chi(\O/_\G)=0$.
\end{fact}

\begin{proof} 
Assume by contradiction that $\Omega/_\Gamma$ is decomposable but $\chi(\Omega/_\Gamma) \neq 0$. By a theorem of Gottlieb \cite[Corollary IV.3]{gottlieb}, the fundamental group of a finite, aspherical polyhedron with non-zero Euler characteristic has trivial centre. But since $\Omega/_\Gamma$ is decomposable, by Proposition 4.4 of Benoist \cite{vey,cd2}, the centre of $\Gamma$ contains a non-trivial free abelian subgroup.
\end{proof}

In particular, since $\chi(X)\neq0$, the convex projective manifold $X$ of Theorem \ref{thm:main} is indecomposable.

\section{The proof of Theorem \ref{thm:main}} \label{sec:proof}

In this section, we prove Theorem \ref{thm:main}. In Section \ref{sec:P_alpha}, we perform convex projective generalised Dehn filling to the ideal hyperbolic rectified 4-simplex $R\subset\Hb^4$, and build mirror polytopes $P_{\alpha}$ combinatorially equivalent\footnote{Two polytopes $Q$ and $Q'$ are \emph{combinatorially equivalent} if the face poset of $Q$ is isomorphic to the face poset of $Q'$.} to the bitruncated 4-simplex. In Section \ref{sec:caprace}, we show that the Coxeter group $W_p=W_{P_{\nicefrac{\pi}{p}}}$ of the Coxeter polytope $P_{\nicefrac{\pi}{p}}$ is relatively hyperbolic. In Sections \ref{sec:M} and \ref{sec:X}, we construct the cusped hyperbolic 4-manifold $M$, by gluing some copies of $R$, and a filling $X$ of $M$, respectively. Finally, in Section \ref{sec:sigma_alpha}, we give $X$ a structure of projective cone-manifold induced by $P_\alpha$, and finish the proof of Theorem \ref{thm:main}.

\subsection{Deforming the rectified 4-simplex} \label{sec:P_alpha}

In this subsection, we obtain a family of projective structures on the orbifold $\mathcal{O}_R$ associated to the ideal hyperbolic rectified 4-simplex $R$ by deforming its original complete hyperbolic structure.

\medskip

Accordingly, the composition of two projective reflections along the facets $F'_i$ and $F'_j$ of $\mathcal{O}_R$, $i\neq j$, (recall the notation from Section \ref{sec:truncated_simplex}) will deform to be conjugate to a projective rotation of angle $2\alpha > 0$. At the original hyperbolic structure on $\mathcal{O}_R$, the facets $F'_i$ and $F'_j$ are parallel, i.e. $\alpha=0$. But, at the deformed projective structure on $\mathcal{O}_R$, new ridges $F'_i\cap F'_j$ can be added to $\mathcal{O}_R$ to form a mirror polytope $P_{\alpha}$ that is combinatorially equivalent to a bitruncated 4-simplex, where $\alpha$ is the dihedral angle of the ridge $F'_i\cap F'_j$. So, the goal is to prove the following:

\begin{proposition}\label{prop:P_theta}
There exists a path $(0,\nicefrac{\pi}{3}]\ni\alpha\mapsto P_\alpha$ of mirror polytopes with the combinatorics of a bitruncated 4-simplex and dihedral angles:
\begin{itemize}
\item $\alpha$ at the filling ridges $F'_i\cap F'_j$,
\item $\nicefrac\pi2$ at the ridges $F_i\cap F'_j$, and
\item $\nicefrac\pi3$ at the ridges $F_i\cap F_j$.
\end{itemize}
Moreover, the limit $P_0$ is the ideal hyperbolic rectified 4-simplex $R$.
\end{proposition}

Before proving Proposition \ref{prop:P_theta}, we begin with some auxiliary lemmas. First, let
$$t_3 = \frac12\left(11 + 9\sqrt2 - 3\sqrt{31 + 22\sqrt2}\right)\approx0.0422 $$
and for $t\in\left[t_3,1\right]$,
$$f(t)=\frac{t(t+2)^3(2t+1)^3}{(t^2+t+1)^2(t^2+7t+1)^2}, \;\; h(t) = 2 + \frac{1}{t} + \frac{ t (t+2)^4 }{(t^2+t+1)(t^2+7t+1)},$$
$$ g_p(t) = \frac{2 t^p (t+2)^p (2t+1)^{3-p}}{(t^2+t+1)(t^2+7t+1)} \;\; \textrm{and} \;\;  \overline{g}_p(t) = \frac{4 f(t)}{g_p(t)}, \;\; \textrm{where} \;\; p=0,1,2,3.$$
The number $t_{3}$ is the unique positive solution less than $1$ satisfying the equation $f(t_3) = \nicefrac{1}{4}= \cos^{2} (\nicefrac{\pi}{3})$. Since $t_3$ is positive, it is easy to check:

\begin{lemma} \label{lem:negative_coefficients}
The functions $f,h,g_p$ and $\overline g_p\colon\left[t_3,1\right]\to\R$ are well-defined and positive.
\end{lemma}

Given $t \in [t_3,1]$, we now consider the matrix
$$C_t =
\left(
\begin{array}{ccccc | ccccc}
2 & -\overline{g}_0(t) & -\overline{g}_1(t) & -\overline{g}_2(t) & -\overline{g}_3(t) & -h(t) & 0 & 0 & 0 & 0 \\
-g_0(t) & 2 & -\overline{g}_0(t) & -\overline{g}_1(t) & -\overline{g}_2(t) & 0 & -h(t) & 0 & 0 & 0 \\
-g_1(t) & -g_0(t) & 2 & -\overline{g}_0(t) & -\overline{g}_1(t) & 0 & 0 & -h(t) & 0 & 0 \\
-g_2(t) & -g_1(t) & -g_0(t) & 2 & -\overline{g}_0(t) & 0 & 0 & 0 & -h(t) & 0 \\
-g_3(t) & -g_2(t) & -g_1(t) & -g_0(t) & 2 & 0 & 0 & 0 & 0 & -h(t) \\
\hline
-2 & 0 & 0 & 0 & 0 & 2 & -\nicefrac1t & -\nicefrac1t & -\nicefrac1t & -\nicefrac1t \\
0 & -2 & 0 & 0 & 0 & -t & 2 & -\nicefrac1t & -\nicefrac1t & -\nicefrac1t \\
0 & 0 & -2 & 0 & 0 & -t & -t & 2 & -\nicefrac1t & -\nicefrac1t \\
0 & 0 & 0 & -2 & 0 & -t & -t & -t & 2 & -\nicefrac1t \\
0 & 0 & 0 & 0 & -2 & -t & -t & -t & -t & 2 \\
\end{array}
\right).
$$

\begin{lemma} \label{lem:rank}
For every $t\in\left[t_3,1\right]$, the matrix $C_t$ is an irreducible Cartan matrix of negative type and of rank $5$.
\end{lemma}

\begin{proof}
It easily follows from Lemma \ref{lem:negative_coefficients} that $C_t$ is an irreducible Cartan matrix. A simple but long computation shows that the rank of $C_t$ is $5$, which is less than $9 = 10 - 1$. Hence, the irreducible Cartan matrix $C_t$ is of negative type by Proposition \ref{prop:cartan_type}.
\end{proof}

\begin{lemma} \label{lem:alpha}
There exists a monotonically decreasing analytic function $\alpha\colon\left[t_3,1\right]\to\R$ satisfying
$$\cos^2\alpha(t)=f(t),\quad\alpha(t_3)=\frac\pi3,\quad\alpha(1)=0.$$
\end{lemma}

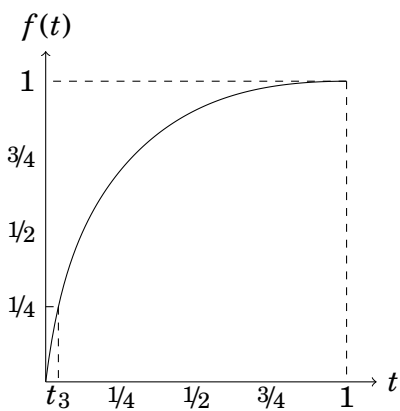
\begin{figure}[!h]
\centering
\begin{tikzpicture}
\draw[->, scale=4] (0,0) -- (1.1,0);
\draw[scale=4] (1.1,0) node[right] {$t$};
\draw [->, scale=4] (0,0) -- (0,1.1);
\draw[scale=4] (0,1.1) node[above] {$f(t)$};
\draw[domain=0:1, scale=4,samples=100] plot [variable=\t] (\t,{\t*(\t+2)^3*(2*\t+1)^3)/((\t^2+\t+1)^2*(\t^2+7*\t+1)^2});
\draw [dashed, scale=4] (0,0.25) -- (0.0422,0.25) node[left] {};
\draw [dashed, scale=4] (0.0422,0.25) -- (0.0422,0) node[right] {};
\draw [scale=4] (0.0422,-0.05) node[] {$t_3$};

\draw [dashed, scale=4] (1,1) -- (0,1) node[left] {$1$};
\draw [dashed, scale=4] (1,1) -- (1,0) node[right] {};
\draw [scale=4] (1,-0.05) node[] {$1$};

\draw [dashed, scale=4] (0,0.75) node[left] {$\nicefrac{3}{4}$};
\draw [dashed, scale=4] (0,0.5) node[left] {$\nicefrac{1}{2}$};
\draw [dashed, scale=4] (0,0.25) node[left] {$\nicefrac{1}{4}$};

\draw [scale=4] (0.25,-0.05) node[] {$\nicefrac{1}{4}$};
\draw [scale=4] (0.5,-0.05) node[] {$\nicefrac{1}{2}$};
\draw [scale=4] (0.75,-0.05) node[] {$\nicefrac{3}{4}$};
\end{tikzpicture}
\caption{\footnotesize The graph of $f(t)$ over the interval $[0,1]$}
\label{func:graph_f}
\end{figure}

\begin{proof}
The derivative $f'$ of $f$ satisfies:
$$f'(t)= - \frac{2(t-1)(t+1)(t+2)^2(2t+1)^2(t^4+2t^3+21t^2+2t+1)}{(t^2+t+1)^3(t^2+7t+1)^3}.$$
Since $f'(t)$ is positive for all $t \in [t_3, 1)$ and is zero for $t=1$, the function $f \colon [t_3,1] \rightarrow \R$ is monotonically increasing with $f(t_3)=\nicefrac{1}{4}$ and $f(1)=1$ (see Figure \ref{func:graph_f}). The lemma now follows easily.
\end{proof}

We are finally ready to prove Proposition \ref{prop:P_theta} by applying Theorems \ref{thm:vinberg_unique} and \ref{thm:vinberg}. Recall the combinatorics of the truncated, rectified and bitruncated 4-simplex, described in Table \ref{table:comb}.  See Figure \ref{fig:link_geom} for a geometric picture of the vertex link of the mirror polytope $P_\alpha$.

\begin{figure}[!h]
\centering
\begin{tabular}{ccc}
\definecolor{ccqqqq}{rgb}{0.8,0,0}
\begin{tikzpicture}[line cap=round,line join=round,>=triangle 45,x=1cm,y=1cm]
\draw [line width=1.5pt] (-1,4)-- (0,3);
\draw [line width=1.5pt] (0,3)-- (1,4);
\draw [line width=1.5pt] (1,4)-- (-1,4);
\draw [line width=1.5pt,color=ccqqqq] (-1,4)-- (-1,0.5);
\draw [line width=1.5pt] (-1,0.5)-- (1,0.5);
\draw [line width=1.5pt] (1,0.5)-- (0,1.5);
\draw [line width=1.5pt] (0,1.5)-- (-1,0.5);
\draw [line width=1.5pt,color=ccqqqq] (1,0.5)-- (1,4);
\draw [line width=1.5pt,color=ccqqqq] (0,3)-- (0,1.5);
\node[draw,circle, inner sep=0.2pt, minimum size=0.2pt] at (0,3.6) {$i'_1$};
\node[draw,circle, inner sep=0.2pt, minimum size=0.2pt] at (0,0.93)  {$i'_2$};
\node[draw,circle, inner sep=0.2pt, minimum size=0.2pt] at (0.43,2.4)  {$i_4$};
\node[draw,circle, inner sep=0.2pt, minimum size=0.2pt] at (-0.57,2.4)  {$i_3$};
\node[draw,circle, inner sep=0.2pt, minimum size=0.2pt] at (-1.58,2.4)  {$i_5$};
\end{tikzpicture}
& $\qquad \qquad$
\definecolor{qqwuqq}{rgb}{0,0.39215686274509803,0}
\definecolor{ccqqqq}{rgb}{0.8,0,0}
\begin{tikzpicture}[line cap=round,line join=round,>=triangle 45,x=1cm,y=1cm]
\draw [line width=1.5pt,color=qqwuqq] (0,1.8)-- (0,-0.2);
\draw [line width=1.5pt] (-1.7320508075688774,-1.2)-- (0,-0.2);
\draw [line width=1.5pt] (0,-0.2)-- (1.7320508075688774,-1.2);
\draw [line width=1.5pt] (1.7320508075688774,-1.2)-- (0,1.8);
\draw [line width=1.5pt] (0,1.8)-- (-1.7320508075688774,-1.2);
\draw [line width=1.5pt, color=ccqqqq] (-1.7320508075688774,-1.2)-- (1.7320508075688774,-1.2);

\node[draw,circle, inner sep=0.2pt, minimum size=0.2pt] at (-0.45,0.22) {$i'_1$};
\node[draw,circle, inner sep=0.2pt, minimum size=0.2pt] at (0.45,0.22) {$i'_2$};
\node[draw,circle, inner sep=0.2pt, minimum size=0.2pt] at (0,-0.75) {$i_3$};
\node[draw,circle, inner sep=0.2pt, minimum size=0.2pt] at (0,-1.6) {$i_4$};
\end{tikzpicture}
\\
\end{tabular}
\caption{\footnotesize The vertex links of a mirror polytope are also mirror polytopes. For $P_\alpha$ we have: $(i)$ a right triangular prism with equilateral bases for the (ideal) rectified 4-simplex $P_0$ (left), and $(ii)$ a tetrahedron for the bitruncated 4-simplex $P_\alpha$, $\alpha\in(0,\frac\pi3]$ (right). The edges of dihedral angle $\nicefrac{\pi}{2}$ (resp. $\nicefrac{\pi}{3}$, resp. $\alpha$) are drawn in black (resp. red, resp. green).}
\label{fig:link_geom}
\end{figure}
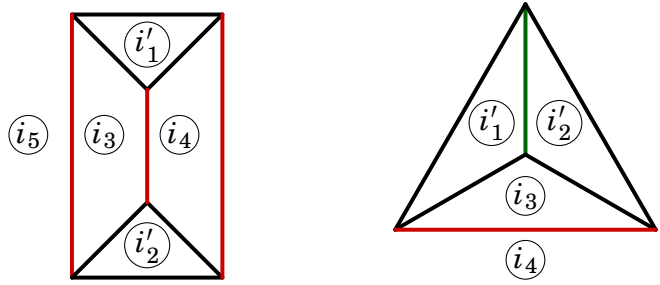

\begin{proof}[Proof of Proposition \ref{prop:P_theta}]
Let us call and order the indices of $C_t$ as $S:=\{ 1',\ldots,5',1,\ldots,5 \}$. Lemmas \ref{lem:rank} and \ref{lem:alpha} together with Theorem \ref{thm:vinberg_unique} imply that there exists a path $[0,\nicefrac{\pi}{3}]\ni \alpha \mapsto P_\alpha$ of mirror polytopes in $\mathbb{S}^4$ with facets $\{ F_s\}_{s\in S}$ and with Cartan matrix $C_t$. The facet $F_{i'}$ of $P_\alpha$ ($i = 1, \dotsc,  5$) is also denoted by $F_i'$.

\medskip

We consider two cases separately: $(i)$ $\alpha \in (0,\nicefrac{\pi}{3}]$ and $(ii)$ $\alpha = 0$.

\medskip

In the case $(i)$, equivalently, $t \in [t_3,1)$, by Remark \ref{rem:ridge} (or Theorem \ref{thm:vinberg}.(\ref{thm:vinberg_faces})), the intersections $F'_{i} \cap F'_{j}$, $F_i \cap F'_{j}$ and $F_i \cap F_j$, $i \neq j$, are ridges of $P_{\alpha}$ and their dihedral angles are $\alpha$, $\nicefrac{\pi}{2}$ and $\nicefrac{\pi}{3}$, respectively, because for example $$ (C_t)_{i' j'} (C_t)_{j' i'} = g_p(t) \overline{g}_p(t) = 4 f(t) = 4 \cos^2\alpha < 4.$$ 
We now claim that $P_{\alpha}$ is combinatorially equivalent to the bitruncated $4$-simplex. For every subset $T = \{ i', j', k, l \} \subset S$ with $\sharp \{i, j, k, l \} = 4$, the submatrix $(C_t)_T$ is the direct sum $(C_t)_{\{i', j'\}} \oplus (C_t)_{\{k, l\}}$ of matrices of positive type, hence $\{ F_{s} \}_{s \in T} \in \sigma(P_\alpha)$ \footnote{See the end of Section \ref{sec:gen_dehn_fill} for the definition of the poset $\sigma(P)$.} and $ \cap_{s \in T} F_{s}$ is a vertex of $P_\alpha$ by Theorem \ref{thm:vinberg}.(\ref{thm:vinberg_faces}). 

That is, if $\Vc$ denotes the set of all subsets 
$\{ F_{i}', F_{j}', F_{k}, F_{l} \}$ with $\sharp \{ i,j,k,l\} = 4$, then $\Vc \subset \sigma(P_{\alpha})$. Let $\hat{S} = \{ F_s \}_{s \in S}$, and let $\Fc$ be the subposet of $2^{\hat{S}}$ defined by:
$$
\Fc := \{ \hat{T}' \in 2^{\hat{S}} \mid \hat{T}' \subset \hat{T} \textrm{ for some } \hat{T} \in \Vc  \}.
$$ 

As in the previous argument, Theorem \ref{thm:vinberg}.(\ref{thm:vinberg_faces}) implies that $\Fc$ is a subposet of $\sigma(P_{\alpha})$. We know, in addition, from Lemma \ref{lem:combi}.(\ref{lem:combi_bitruncated}) and Figure \ref{fig:link_comb} that the poset $\Fc$ is dual to the face poset of the bitruncated 4-simplex. It is a well-known fact (e.g. \cite[Exercise 1.1.20]{BP15}) that if two polytopes $Q$ and $Q'$ are of same dimension and the face poset of $Q'$ is a subposet of the face poset of $Q$, then $Q$ is combinatorially equivalent to $Q'$. As a result, the polytope $P_{\alpha}$ is combinatorially equivalent to the bitruncated $4$-simplex, as claimed.

\medskip

In the case $(ii)$, equivalently, $t=1$, we claim that $P_{0}$ is combinatorially equivalent to the rectified 4-simplex. For every subset $U = \{ i', j', k, l, m \} \subset S$ with $\sharp \{i, j, k, l, m \} = 5$, the submatrix $(C_1)_U$ is the direct sum $(C_1)_{\{i', j'\}} \oplus (C_1)_{\{k, l, m\}}$ of matrices of zero type and the rank of $(C_1)_U$ is $3$. Hence $\{ F_{s} \}_{s \in U} \in \sigma(P_0)$ and $ \cap_{s \in U} F_{s}$ is a vertex of $P_0$ by Theorem \ref{thm:vinberg}.(\ref{thm:vinberg_faces_zero}). Furthermore, for every subset $U' = \{ i', j, k \} \subset S$ with $\sharp \{i, j, k \} = 3$, the submatrix $(C_1)_{U'}$ is of positive type. Hence $\{ F_{s} \}_{s \in U'} \in \sigma(P_0)$ and $ \cap_{s \in U'} F_{s}$ is an edge of $P_0$ by Theorem \ref{thm:vinberg}.(\ref{thm:vinberg_faces}). Then, as in the proof of case $(i)$, using Lemma \ref{lem:combi}.(\ref{lem:combi_rectified}) and Figure \ref{fig:link_comb}, we may conclude that the polytope $P_{0}$ is combinatorially equivalent to the rectified $4$-simplex, as claimed. Finally, a simple computation shows that the Cartan matrix $C_1$ is equivalent to a symmetric matrix of signature $(4,1)$, and therefore the polytope $P_{0}$ (without vertices) may be identified with the ideal hyperbolic rectified $4$-simplex $R$.
\end{proof}

\begin{remark} \label{rem:symmmetry}
The symmetry group of the mirror polytope $P_\alpha$, $\alpha \in [0,\nicefrac{\pi}{3}]$, is of order $\geq 5$. For, if we set $\hat{Q} = \bigl(\begin{smallmatrix}
Q & 0 \\ 0 & Q
\end{smallmatrix} \bigr)$ with
$$ 
Q =
\left(
\begin{array}{ccccc}
0 & 0 & 0 & 0 & 1 \\
1 & 0 & 0 & 0 & 0 \\
0 & 1 & 0 & 0 & 0 \\
0 & 0 & 1 & 0 & 0 \\
0 & 0 & 0 & 1 & 0 \\
\end{array}
\right),
$$
then $\hat{Q}$ is a permutation matrix of order $5$ and $\hat{Q} C_t {\hat{Q}}^{-1}$ is equivalent to $C_t$. 
\end{remark}

\begin{remark}
A (bit more complicated) computation reveals that the deformation space of projective structures on the orbifold $\mathcal{O}_R$ associated to the ideal hyperbolic rectified $4$-simplex $R$ is six-dimensional. In other words, one can find a six-parameter family of Cartan matrices (of rank $5$) which correspond to projective structures on $\mathcal{O}_R$. But, we choose a particular one-parameter family of Cartan matrices having symmetry as described in Remark \ref{rem:symmmetry} in order to simplify the computation.
\end{remark}

We end the subsection with the following.

\begin{corollary} \label{cor:coxeter}
For every integer $p\geq3$, the mirror polytope $P_{\nicefrac{\pi}{p}}$ is a Coxeter polytope. Moreover, if $\Gamma_p$ denotes the subgroup of $\mathrm{SL}^{\pm}_5\R$ generated by the associated reflections, then the $\Gamma_p$-orbit of $P_{\nicefrac{\pi}{p}}$ is a properly convex domain $\Omega_p$ of $\mathbb{S}^4$, i.e. it is divisible by $\Gamma_p$.
\end{corollary}

\begin{proof}
It is obvious that the mirror polytope $P_{\nicefrac{\pi}{p}}$ is a Coxeter polytope. For any subset $T = \{i',j',k,l\} \subset S$ with $\sharp \{ i,j,k,l\} = 4$, the standard subgroup $W_T$ of the Coxeter group $W_S$ of $P_{\nicefrac{\pi}{p}}$ is isomorphic to $D_p\times D_3$, where $D_m$ is the dihedral group of order $2m$, hence it is a finite group. In the proof of Proposition \ref{prop:P_theta}, we show that every vertex of $P_{\nicefrac{\pi}{p}}$ corresponds to a subset $\{i',j',k,l\} \subset S$ with $\sharp \{ i,j,k,l\} = 4$, so by Theorem \ref{theo_vinberg}, the $\Gamma_p$-orbit of $P_{\nicefrac{\pi}{p}}$ is a divisible convex domain of $\mathbb{S}^4$.
\end{proof}

\subsection{The Coxeter group $W_p$}\label{sec:caprace}

The goal of this subsection is to show that the Coxeter group $W_{p}$ of $P_{\nicefrac\pi p}$ is relatively hyperbolic with respect to a collection of virtually abelian subgroups of rank $2$.
To do so, we need to analyse the Coxeter diagram of $W_{p}$ in Figure \ref{diag_Wk}, and to use Theorem \ref{moussong_caprace} together with the (complete) list of the irreducible spherical and affine Coxeter groups in Appendix \ref{classi_diagram}.

\begin{figure}[!h]
\centering

\definecolor{ccqqqq}{rgb}{0.2,0.2,0.7}
\definecolor{qqwuqq}{rgb}{0,0,0}
\definecolor{qqqqff}{rgb}{0,0.39215686274509803,0}
\definecolor{uuuuuu}{rgb}{0.26666666666666666,0.26666666666666666,0.26666666666666666}
\begin{tikzpicture}[line cap=round,line join=round,>=triangle 45,x=1cm,y=1cm,scale=0.55]
\clip(-9.5,-9) rectangle (9.5,9.5);
\draw [line width=1.5pt,color=qqqqff] (0,6)-- (5.706339097770921,1.8541019662496832);
\draw [line width=1.5pt,color=qqqqff] (5.706339097770921,1.8541019662496832)-- (3.5267115137548375,-4.854101966249685);
\draw [line width=1.5pt,color=qqqqff] (3.5267115137548375,-4.854101966249685)-- (-3.5267115137548393,-4.854101966249684);
\draw [line width=1.5pt,color=qqqqff] (-3.5267115137548393,-4.854101966249684)-- (-5.706339097770921,1.8541019662496852);
\draw [line width=1.5pt,color=qqqqff] (-5.706339097770921,1.8541019662496852)-- (0,6);
\draw [line width=1.5pt,color=qqwuqq] (0,3)-- (-2.8531695488854605,0.9270509831248426);
\draw [line width=1.5pt,color=qqwuqq] (-2.8531695488854605,0.9270509831248426)-- (-1.7633557568774196,-2.427050983124842);
\draw [line width=1.5pt,color=qqwuqq] (-1.7633557568774196,-2.427050983124842)-- (1.7633557568774187,-2.4270509831248424);
\draw [line width=1.5pt,color=qqwuqq] (1.7633557568774187,-2.4270509831248424)-- (2.853169548885461,0.9270509831248418);
\draw [line width=1.5pt,color=qqwuqq] (2.853169548885461,0.9270509831248418)-- (0,3);
\draw [line width=1.5pt,color=qqwuqq] (0,3)-- (1.7633557568774187,-2.4270509831248424);
\draw [line width=1.5pt,color=qqwuqq] (1.7633557568774187,-2.4270509831248424)-- (-2.8531695488854605,0.9270509831248426);
\draw [line width=1.5pt,color=qqwuqq] (-2.8531695488854605,0.9270509831248426)-- (2.853169548885461,0.9270509831248418);
\draw [line width=1.5pt,color=qqwuqq] (0,3)-- (-1.7633557568774196,-2.427050983124842);
\draw [line width=1.5pt,color=qqwuqq] (-1.7633557568774196,-2.427050983124842)-- (2.853169548885461,0.9270509831248418);
\draw [line width=1.5pt,color=ccqqqq] (-5.706339097770921,1.8541019662496852)-- (-2.8531695488854605,0.9270509831248426);
\draw [line width=1.5pt,color=ccqqqq] (0,6)-- (0,3);
\draw [line width=1.5pt,color=ccqqqq] (5.706339097770921,1.8541019662496832)-- (2.853169548885461,0.9270509831248418);
\draw [line width=1.5pt,color=ccqqqq] (3.5267115137548375,-4.854101966249685)-- (1.7633557568774187,-2.4270509831248424);
\draw [line width=1.5pt,color=ccqqqq] (-1.7633557568774196,-2.427050983124842)-- (-3.5267115137548393,-4.854101966249684);
\draw [shift={(1.763355756877419,0.5729490168751576)},line width=1.50pt,color=qqqqff]  plot[domain=-1.2566370614359172:1.8849555921538756,variable=\t]({1*5.706339097770921*cos(\t r)+0*5.706339097770921*sin(\t r)},{0*5.706339097770921*cos(\t r)+1*5.706339097770921*sin(\t r)});
\draw [shift={(1.0898137920080408,-1.5)},line width=1.50pt,color=qqqqff]  plot[domain=-2.5132741228718345:0.6283185307179586,variable=\t]({1*5.706339097770921*cos(\t r)+0*5.706339097770921*sin(\t r)},{0*5.706339097770921*cos(\t r)+1*5.706339097770921*sin(\t r)});
\draw [shift={(-1.0898137920080417,-1.5)},line width=1.50pt,color=qqqqff]  plot[domain=2.5132741228718345:5.654866776461628,variable=\t]({1*5.706339097770921*cos(\t r)+0*5.706339097770921*sin(\t r)},{0*5.706339097770921*cos(\t r)+1*5.706339097770921*sin(\t r)});
\draw [shift={(-1.7633557568774194,0.5729490168751581)},line width=1.50pt,color=qqqqff]  plot[domain=1.2566370614359172:4.39822971502571,variable=\t]({1*5.70633909777092*cos(\t r)+0*5.70633909777092*sin(\t r)},{0*5.70633909777092*cos(\t r)+1*5.70633909777092*sin(\t r)});
\draw [shift={(0,1.8541019662496843)},line width=1.50pt,color=qqqqff]  plot[domain=0:3.141592653589793,variable=\t]({1*5.706339097770921*cos(\t r)+0*5.706339097770921*sin(\t r)},{0*5.706339097770921*cos(\t r)+1*5.706339097770921*sin(\t r)});
\draw (7.343586957508095,3.0875495754288003) node[anchor=north west] {$p$};
\draw (0,8.3) node {$p$};
\draw (-8.3,3.08) node[anchor=north west] {$p$};
\draw (5.0566580025686605,-5.693671219498904) node[anchor=north west] {$p$};
\draw (-5.043944881747179,-6.030846642342539) node[anchor=north west] {$p$};
\draw (4.7488021817114285,-1.2077721155792422) node[anchor=north west] {$p$};
\draw (2.7257496446496203,4.758766888653772) node[anchor=north west] {$p$};
\draw (-3.8,4.753529674748951) node[anchor=north west] {$p$};
\draw (-5.6,-1.0758339066404286) node[anchor=north west] {$p$};
\draw (-0.22087035498388263,-4.887382164872821) node[anchor=north west] {$p$};
\draw (0.5,4.5) node {$\infty$};
\draw (4,1.7) node {$\infty$};
\draw (-4,1.7) node {$\infty$};
\draw (3,-3.4) node {$\infty$};
\draw (-3,-3.4) node {$\infty$};
\begin{scriptsize}
\draw [fill=white] (0,3) circle (12pt) node {$1$};
\draw [fill=white] (0,6) circle (12pt) node {$1'$};
\draw [fill=white] (-5.706339097770921,1.8541019662496852) circle (12pt) node {$5'$};
\draw [fill=white] (-3.5267115137548393,-4.854101966249684) circle (12pt) node {$4'$};
\draw [fill=white] (3.5267115137548375,-4.854101966249685) circle (12pt) node {$3'$};
\draw [fill=white] (5.706339097770921,1.8541019662496832) circle (12pt) node {$2'$};
\draw [fill=white] (-2.8531695488854605,0.9270509831248426) circle (12pt) node {$5$};
\draw [fill=white] (-1.7633557568774196,-2.427050983124842) circle (12pt) node {$4$};
\draw [fill=white] (1.7633557568774187,-2.4270509831248424) circle (12pt) node {$3$};
\draw [fill=white] (2.853169548885461,0.9270509831248418) circle (12pt) node {$2$};
\end{scriptsize}
\end{tikzpicture}
\caption{\footnotesize The Coxeter diagram of the Coxeter group $W_p$ of $P_{\nicefrac\pi p}$}
\label{diag_Wk}
\end{figure}
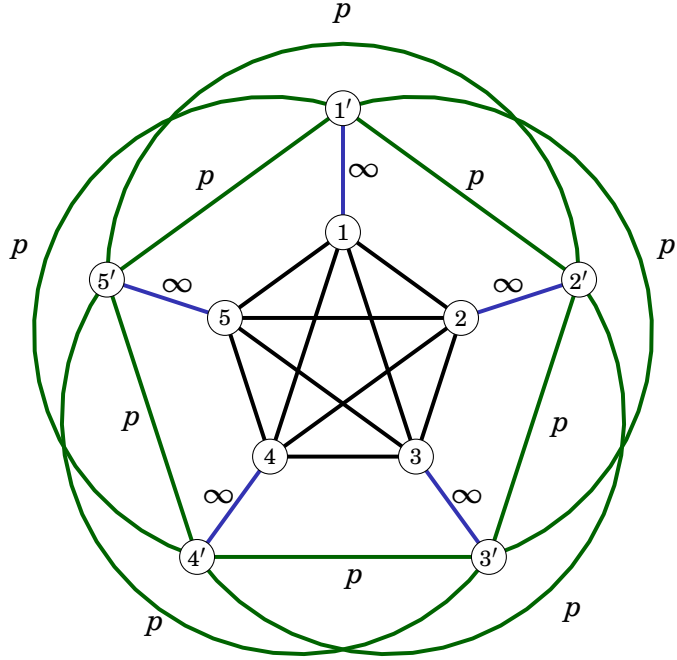

We denote by $\Tc_{p}$ the following collection of subsets of $S=\{ 1', \ldots,5', 1, \ldots,5 \}$:
\begin{itemize}
	\item In the case of $p \geq 4$, the collection $\Tc_{p}$ consists of all subsets $\{i',j',k,l,m\}$ of $S$ with $\sharp \{ i,j,k,l,m\} = 5$, so the cardinality of $\Tc_p$, $p \geq 4$, is $10$.
	\item In the case of $p = 3$, the collection $\Tc_{3}$ consists of all subsets $\{i',j',k,l,m\}$ and $\{i',j',k',l,m\}$ of $S$ with $\sharp \{ i,j,k,l,m\} = 5$, so the cardinality of $\Tc_3$ is $20$.
\end{itemize}
For each $U \in \Tc_p$, the standard subgroup $(W_p)_U$ of $W_p$ is isomorphic to $\widetilde{A}_{2} \times I_2(p)$ (see Appendix \ref{classi_diagram}), hence it is virtually isomorphic to $\Z^{2}$.

\begin{proposition}\label{prop:rel_hyp}
The Coxeter group $W_p$ is relatively hyperbolic with respect to the collection of subgroups $ \{ (W_p)_U \mid U \in \Tc_p \}$, in particular, a collection of virtually abelian subgroups of rank $2$.
\end{proposition}

\begin{proof}
We only prove it for the case $p=3$; the argument is similar for other cases $p \geq 4$. Thanks to Theorem \ref{moussong_caprace}, we just need to carefully analyse the Coxeter diagram of $W_p$ in Figure $\ref{diag_Wk}$, using the list of irreducible spherical and affine Coxeter groups in Appendix \ref{classi_diagram}. 

\medskip

First, the condition (\ref{thm:Caprace1}) holds because all irreducible affine subsets $U \subset S$ of rank $\geq 3$ are $\{i',j',k'\}$ and $\{i,j,k\}$ with $\sharp \{i,j,k\} = 3$. Second, the condition (\ref{thm:Caprace2}) holds because there does not exist a pair of irreducible non-spherical subsets $S_{1},S_{2}$ of $S$ with $S_1 \perp S_2$. Third, the condition (\ref{thm:Caprace3}) holds because for every pair $T, T' \in \Tc_3$ with $T \neq T'$, the intersection $T \cap T'$ is a subset of $\{i',j',k,l\}$ with $\sharp \{i,j,k,l\} =4$. Finally, the condition (\ref{thm:Caprace4}) holds because for every $T \in \Tc_3$, there exists only one irreducible non-spherical subset $U \subset T$, which is either $\{i',j',k'\}$ or $\{i,j,k\}$ with $\sharp \{i,j,k\} = 3$, and $T = U\, \sqcup \, U^{\perp}$.
\end{proof}

\begin{remark}
The Coxeter group $W_p$ of $P_{\nicefrac{\pi}{p}}$ is in fact a finite-index subgroup of the following Coxeter group $\hat{W}_p$:
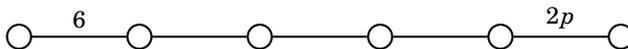
\begin{center}
\begin{tikzpicture}[thick,scale=0.8, every node/.style={transform shape}]
\node[draw, circle, inner sep=1pt, minimum size=4mm] (1) at (0,0)  {};
\node[draw, circle, inner sep=1pt, minimum size=4mm] (2) at (2,0)  {};
\node[draw, circle, inner sep=1pt, minimum size=4mm] (3) at (4,0)  {};
\node[draw, circle, inner sep=1pt, minimum size=4mm] (4) at (6,0)  {};
\node[draw, circle, inner sep=1pt, minimum size=4mm] (5) at (8,0)  {};
\node[draw, circle, inner sep=1pt, minimum size=4mm] (6) at (10,0)  {};

\draw (1) -- (2) node[above,midway]   {$6$};
\draw (2) -- (3) node[above,midway]   {};
\draw (3) -- (4) node[above,midway]   {};
\draw (4) -- (5) node[above,midway]   {};
\draw (5) -- (6) node[above,midway]   {$2p$};
\end{tikzpicture}
\end{center}
It is easier to verify the conditions (\ref{thm:Caprace1}) -- (\ref{thm:Caprace4}) of Theorem \ref{moussong_caprace} for the Coxeter group $\hat{W}_p$.
\end{remark}

\subsection{The cusped hyperbolic manifold} \label{sec:M}

In this subsection, we build the cusped hyperbolic manifold $M$ of Theorem \ref{thm:main}. A reader who is not familiar with hyperbolic manifolds (with totally geodesic boundary) could consult \cite[Chapter 3]{Mbook}.

\begin{figure}[!h]
\begin{center}
\begin{tikzpicture}[thick,scale=0.7, every node/.style={transform shape}]
\node[draw, circle, fill=black, inner sep=1pt, minimum size=1mm] (1) at (2,0)  {};
\node[draw, circle, fill=black, inner sep=1pt, minimum size=1mm] (2) at (6,0)  {};
\node[draw, circle, fill=black, inner sep=1pt, minimum size=1mm] (3) at (8,3.464)  {};
\node[draw, circle, fill=black, inner sep=1pt, minimum size=1mm] (4) at (6,6.928)  {};
\node[draw, circle, fill=black, inner sep=1pt, minimum size=1mm] (5) at (2,6.928)  {};
\node[draw, circle, fill=black, inner sep=1pt, minimum size=1mm] (6) at (0,3.464)  {};

\draw (0.8,5.5) node[below] {1};
\draw (4,0) node[below] {1};
\draw (7.2,5.5) node[below] {1};

\draw (0.8,2.0) node[below] {2};
\draw (4,7.4) node[below] {2};
\draw (7.2,2.0) node[below] {2};

\draw (4.8,3.95) node[below] {3};
\draw (6.15,5.6) node[below] {3};
\draw (1.85,1.6) node[below] {3};

\draw (3.15,6.75) node[below] {4};
\draw (3.6,2.8) node[below] {4};
\draw (4.7,0.8) node[below] {4};

\draw (1.0,4.6) node[below] {5};
\draw (3.3,4.6) node[below] {5};
\draw (7.0,3.0) node[below] {5};

\draw[line width=0.1mm] (1) -- (2) node[above,midway]   {};
\draw[line width=0.1mm] (1) -- (3) node[above,midway]   {};
\draw[line width=0.1mm] (1) -- (4) node[above,midway]   {};
\draw[line width=0.1mm] (1) -- (5) node[above,midway]   {};
\draw[line width=0.1mm] (1) -- (6) node[above,midway]   {};
\draw[line width=0.1mm] (2) -- (3) node[above,midway]   {};
\draw[line width=0.1mm] (2) -- (4) node[above,midway]   {};
\draw[line width=0.1mm] (2) -- (5) node[above,midway]   {};
\draw[line width=0.1mm] (2) -- (6) node[above,midway]   {};
\draw[line width=0.1mm] (3) -- (4) node[above,midway]   {};
\draw[line width=0.1mm] (3) -- (5) node[above,midway]   {};
\draw[line width=0.1mm] (3) -- (6) node[above,midway]   {};
\draw[line width=0.1mm] (4) -- (5) node[above,midway]   {};
\draw[line width=0.1mm] (4) -- (6) node[above,midway]   {};
\draw[line width=0.1mm] (5) -- (6) node[above,midway]   {};
\end{tikzpicture}
\end{center}
\caption{\footnotesize The edge-labelled complete graph $K_6$}
\label{fig:K6}
\end{figure}
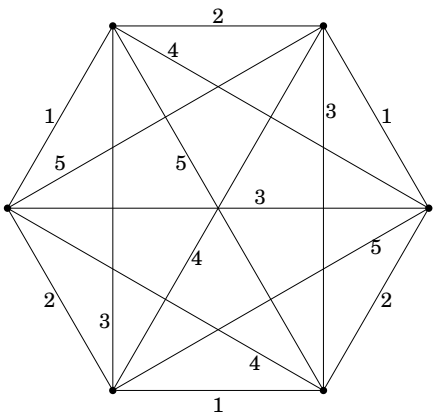

\medskip

We first recall the construction of a building block $B$ by Kolpakov and Slavich \cite{KS}. Consider the complete graph $K_6$ on 6 vertices with its edges labelled by numbers in $\{1,\ldots,5\}$ so that adjacent edges have distinct labels (see Figure \ref{fig:K6}).
For each vertex of $K_6$, take a copy of the ideal hyperbolic rectified 4-simplex $R\subset\Hb^4$, described in Section \ref{subsec:hyperbolic_rectified}. If two vertices of $K_6$ are connected by an edge of label $i\in\{1,\ldots,5\}$, then we glue together the facets $F_i$ of the two corresponding copies of $R$ using the identity as a gluing map.

\begin{proposition}[{\cite[Section 3]{KS}}]\label{prop:Building_Block}
Let $B$ be the building block constructed above.
\begin{itemize}

\item\label{prop:B_mfd} The space $B$ is a non-orientable, complete, finite-volume hyperbolic $4$-manifold with non-compact totally geodesic boundary.

\item\label{prop:B_bdry} The boundary $\partial B$ of $B$ has exactly 5 connected components $\partial_1,\ldots,\partial_5$, each tessellated by the facets $F'_i$ of the copies of $R$ in $B$.

\item\label{prop:B_cusps} The hyperbolic manifold $B$ has exactly 10 cusps $C_{ij}$, $i<j\in\{1,\ldots,5\}$. Each cusp section $S_{ij}$ is diffeomorphic to $K\times [0,1]$, where $K$ denotes the Klein bottle. One boundary component of $S_{ij}$ is contained in $\partial_i$, and the other in $\partial_j$.
\end{itemize}
\end{proposition}

Now, consider again the edge-labelled graph $K_6$ in Figure \ref{fig:K6}. For each vertex of $K_6$, take a copy of $B$.
If two vertices of $K_6$ are joined by an edge of label $i\in\{1,\ldots,5\}$, then we glue together the boundary components $\partial_i$ of the two corresponding copies of $B$ using the identity as a gluing map.
Let us call $M'$ the resulting cusped hyperbolic manifold (without boundary), and $M$ its orientable double cover.

\begin{proposition}\label{prop:M}
The orientable hyperbolic $4$-manifold $M$ has exactly 10 cusps, each with section diffeomorphic to the $3$-torus.
\end{proposition}
\begin{proof}
We obtain the cusps of $M'$ by gluing together the cusps of the copies of $B$, and the gluing maps are induced by the identity. Hence, each cusp section of $M'$ must be diffeomorphic to $K\times\Sb^1$. By construction, for each pair $i<j\in\{1,\ldots,5\}$, the simple cycle (of length $6$) in the graph $K_6$ with edges labelled alternately by $i$ and $j$ corresponds to a cusp of $M'$, thus $M'$ has exactly 10 cusps.
Since the cusps of $M'$ are non-orientable, each cusp section of $M$ is the orientable double covering of a cusp section of $M'$. In particular, $M$ has precisely 10 cusps, each with $3$-torus section.
\end{proof}

\begin{remark} \label{rem:orb_cov_M}
The natural map $M\to \mathcal{O}_R$ is an orbifold covering of degree $6\cdot6\cdot2=72$, where $\mathcal{O}_R$ denotes the orbifold associated to $R$.
\end{remark}

\subsection{Topology of the filling} \label{sec:X}
In this subsection, we build the manifold $X$ of Theorem \ref{thm:main} topologically. The desired (singular) geometric structures on $X$ will be given in the next subsection.

\medskip

Let $Q \subset \R^4$ be a bitruncated $4$-simplex. We define $X$ in the same way as the manifold $M$ (see Section \ref{sec:M}), but substituting the ideal rectified simplex $R$ with the bitruncated simplex $Q$.

\begin{proposition}
The space $X$ is a closed, orientable, smooth $4$-manifold, containing 10 pairwise disjoint embedded $2$-tori whose complement is diffeomorphic to the cusped hyperbolic manifold $M$.
\end{proposition}

\begin{proof}
To prove that the polyhedral complex $X$ is a smooth manifold, it suffices to show that the link of each vertex is homeomorphic to the 3-sphere. Recall that the polytope $Q$ is vertex transitive. The vertex link $L$ of $Q$ is the tetrahedron in the right of  Figure \ref{fig:link_comb}.

\medskip

When we glue the first 6 copies of $Q$ to form $B$, for each vertex of $Q$, 6 copies of $L$ are glued cyclically around one edge, to form a polyhedral complex $L'$ homeomorphic to the closed 3-disc $D^3$.
Thus, $B$ is a 4-manifold with boundary.
Note that $R$ is homeomorphic to the complement $Q \smallsetminus \bigcup_{i<j}F'_j\cap F'_j$ of the filling ridges; in particular, $B$ is non-orientable (recall Proposition \ref{prop:Building_Block}).
When we glue 6 copies of $B$ to get $X'$, for each vertex of the complex, 6 copies of $L'$ are glued cyclically around a circle $C\subset\partial L'$.
The resulting polyhedral complex is clearly homeomorphic to the 3-sphere, so $X'$ is a manifold. Its orientation cover $X$ is thus a manifold.

\medskip

Now, the cusped hyperbolic manifold $M$ is diffeomorphic to the interior of a compact manifold $\overline M$ with boundary $\partial\overline M$ consisting of 10 3-tori (recall Proposition \ref{prop:M}).
Let $\Sigma\subset X$ be the union of the copies of the filling ridges of $Q$. Clearly, $M$ is diffeomorphic to $X\smallsetminus\Sigma$. Moreover, $\partial\overline M$ is diffeomorphic to the boundary of a regular neighbourhood of $\Sigma$ in $X$.
Thus, $\Sigma$ consists of 10 2-tori.
\end{proof}

We conclude the subsection with some additional information about $X$.

\begin{remark}\label{rem:chi_X}
The 4-manifold $X$ has Euler characteristic $\chi(X)=12$. Indeed, $\mathcal O_R$ has orbifold Euler characteristic $\chi^{\mathrm{orb}}(\mathcal O_R)=\frac16$ \cite[Appendix 1]{KS}, and the covering $M\to \mathcal O_R$ has degree 72 (recall Remark \ref{rem:orb_cov_M}), so $\chi(M)=12$. Since $\chi(\partial\overline M)=0$ and $\chi(\Sigma)=0$, we have $\chi(X)=\chi(M)=12$.
\end{remark}

\begin{remark}\label{rem:6-filling}
The manifold $X$ is a filling of $M$, which has a maximal cusp section $S$ such that each filling curve in $S$ has length $6$. The reason is that the maximal, and maximally symmetric, cusp section of $\mathcal O_R$ consists of 10 Euclidean prisms, each with all edges of length $1$ \cite[Section 3.2]{KS}. Each of the filling curves of $X$ is made of $6$ copies of the height of such prism. Since $6<2\pi$, one cannot apply directly the Gromov-Thurston $2\pi$ theorem to conclude, for instance, that $X$ is aspherical. We will see later that, being convex projective, $X$ is in fact aspherical. 
\end{remark}

\begin{remark}
The 6 theorem, independently obtained by Agol \cite{agol} and Lackenby \cite{lackenby}, shows that a filling of a cusped hyperbolic 3-manifold carries a hyperbolic structure as soon as the filling curves are of length \textit{strictly} greater than $6$. This is thus an improvement of Gromov and Thurston's $2\pi$  theorem in dimension three. It is an open question whether or not it is possible to generalise the 6 theorem to higher dimension as follows: ``The fundamental group of a filling of a cusped hyperbolic $n$-manifold is relatively hyperbolic with respect to the collection of subgroups associated to the inserted $(n-2)$-submanifolds, as soon as the filling curves are of length $>6$.''

Note that in dimension $n=3$ the bound of 6 is sharp, as shown by Agol \cite[Section 7]{agol}. Remark \ref{rem:6-filling} shows that the same bound would be sharp in dimension $n=4$. 
\end{remark}

\subsection{Cone-manifold structures on the filling} \label{sec:sigma_alpha}

We now conclude the proof of Theorem \ref{thm:main}.
Let $\Sigma\subset X$ be the union of the copies of the filling ridges $F'_i\cap F'_j$. We first show item \eqref{item:thmA_cone}, giving a path of projective cone-manifold structures on $X$ with the desired properties.

\medskip

In Proposition \ref{prop:P_theta}, we built a path of structures $(0,\frac\pi3]\ni\alpha\mapsto P_\alpha$ of mirror polytope on the bitruncated simplex $Q$. Since the manifold $X$ is built by pairing the facets of some copies of $Q$ through the map induced by the identity, for each $\alpha$ we have a well-defined projective structure on the complement in $X$ of the ridges of the copies of $Q$.
Indeed, the projective structures of the copies of $P_\alpha\smallsetminus\bigcup(\mbox{ridges})$ match well via the projective reflections associated to the facets of $P_\alpha$. We want to show that this projective structure extends to a projective structure on $X\smallsetminus\Sigma$, and on the whole $X$ as projective cone-manifold structures with cone angle $\theta=6\cdot\alpha$ along $\Sigma$. 

\medskip

The link $L$ of a vertex of $P_\alpha$ is a mirror tetrahedron. Its non-right dihedral angles are $\alpha$ and $\frac\pi3$ along its two opposite edges $\{i'_1,i'_2\}$ and $\{i_3,i_4\}$, respectively (see Figure \ref{fig:link_geom}--right).
Recall Section \ref{sec:cone-mfds} about projective cone-manifolds.
Some copies of $L$ are glued together in $X$ to form a 3-sphere, which we now show to be the projective cone 3-manifold $\Sb^1*C_{\theta}$, where $\theta=6\cdot\alpha$.

\medskip

The manifold $X$ was built in three steps.
At the first step (when gluing 6 copies of $Q$ to build the block $B$), 6 copies of $L$ are glued cyclically along its edge $\{i_3,i_4\}$.
The resulting space $L'$ is the intersection of two half-spaces of $\Sb^3$ with \lq\lq dihedral angle\rq\rq\footnote{Since to each hyperplane is associated a reflection, it makes sense to talk about the dihedral angle, even if $L'$ is not properly convex.} $\alpha$. 
At the second step (when gluing 6 copies of $B$ to build $X'$), 6 copies of $L'$ are glued cyclically along its edge. Indeed, the cycle in the graph $K_6$ with edges labelled alternately by $i_1$ and $i_2$ has length 6.
The resulting space is thus $\Sb^1*C_{\theta}$.
The third step is just the orientation double covering $X\to X'$, so the local structure of $X$ is the same as that of $X'$.

\medskip

The union of the copies of the filling ridge $F'_{i_1}\cap F'_{i_2}$ of $Q$ in $X$ is a component $T$ of $\Sigma$ (a 2-torus).
The holonomy of a meridian of $T$ in $X$ is $(\sigma'_{i_1}\sigma'_{i_2})^3$, where $\sigma'_i$ is the projective reflection (depending on $\alpha$) associated to the facet $F'_i$ of the mirror polytope $P_\alpha$.

\medskip

Since $(X,\Sigma)$ is locally modeled on $(\Sb^2*C_\theta,\Sb^2)$, we have a projective cone structure $\sigma_\theta$ for each $\theta\in(0,2\pi]$. Since $\Sb^2*C_{2\pi}=\Sb^4$, the projective structure $\sigma_{2 \pi}$ is non-singular, and each component of $\Sigma$ is totally geodesic in $X$ (see Remark \ref{rem:nonsingular}). The associated path of projective cone-manifold structures $\theta\mapsto\sigma_\theta$ on $X$ is analytic because the path of Cartan matrices $t\mapsto C_t$ is analytic.

\medskip

We have shown item \eqref{item:thmA_cone} of Theorem \ref{thm:main}. Clearly, item \eqref{item:thmA_hyperbolic} follows from Proposition \ref{prop:M}. Item \eqref{item:thmA_Euler} is shown in Remark \ref{rem:chi_X}, and item \eqref{item:thmA_curve} in Remark \ref{rem:6-filling}.

\medskip

Let us now fix an integer $p\geq3$.
By Corollary \ref{cor:coxeter}, the mirror polytope $P_{\nicefrac\pi p}$ is a convex projective orbifold.
If moreover $p=3m$ (i.e. $\theta=\nicefrac{2\pi}m$), the natural map $X\to P_{\nicefrac\pi p}$ gives $X$ a structure of convex projective orbifold $\Omega_m/_{\Gamma_m}$ with cone structure $\sigma_{\nicefrac{2\pi}m}$. We have thus shown the first part of item \eqref{item:thmA_orb}.

\begin{remark}
We already know that the convex projective manifold $X$ is indecomposable. 
Another way to see this is as follows. Since $X$ covers the Coxeter polytope $P_{\nicefrac\pi 3}$, and the Coxeter group $W_{3}$ acts strongly irreducibly on $\R^5$ because the Cartan matrix of $P_{\nicefrac\pi 3}$ is indecomposable and $W_3$ is not virtually abelian (see e.g. \cite[Theorem 2.18]{cox_in_hil}), so does $\Gamma_1\cong\pi_1(X)$.
\end{remark}

Let $T_1,\ldots,T_{10}$ be the components of $\Sigma$. Being totally geodesic for each $T_i$, the natural map $\pi_1T_i\to\Gamma_m$ induced by the inclusion $T_i\subset X$ is injective.
For $m \geqslant 2$, the group $\Gamma_m$ is relatively hyperbolic with respect to $\{\pi_1 T_i\}_i$. Indeed, by Proposition \ref{prop:rel_hyp}, the Coxeter group $W_{3m}$, which is the orbifold fundamental group of $P_{\nicefrac\pi{3m}}$, is relatively hyperbolic with respect to the collection $\{ (W_{3m})_U \mid U \in \Tc_{3m} \}$ which corresponds to the fundamental groups of the $T_1,\ldots,T_{10}$. The proof of item \eqref{item:thmA_orb} is complete.

\medskip

It remains to show item \eqref{item:thmA_rel}. For $m=1$ (i.e $\alpha=\nicefrac\pi 3$ and $\theta = 2 \pi$) there is another collection of totally geodesic tori $T'_1,\ldots,T'_{10}$ tiled by the ridges $F_{i} \cap F_{j}$. Being totally geodesic, also $T'_i$ is $\pi_1$-injective. This time, the group $\Gamma_1\cong\pi_1X$ is relatively hyperbolic with respect to $\{\pi_1 T_i,\,\pi_1 T'_i\}_i$. Indeed, by Proposition \ref{prop:rel_hyp}, the Coxeter group $W_{3}$, which is the orbifold fundamental group of $P_{\nicefrac\pi3}$, is relatively hyperbolic with respect to the collection $\{ (W_{3})_U \mid U \in \Tc_{3} \}$ which corresponds to the fundamental groups of the $T_1,\ldots,T_{10}$ and $T'_1,\ldots,T'_{10}$. Finally, for every $T\in\{T_1,\ldots,T_{10}\}$ and $T'\in\{T'_1,\ldots,T'_{10}\}$, the tori $T$ and $T'$ are transverse (sometimes $T\cap T'=\varnothing$) in $X$, since the ridges $F'_i\cap F'_j$ and $F_k\cap F_\ell$ do so in $\Sb^4$. Also, in the universal cover $\Omega$, the lifts of $T$ and $T'$ are transverse.

\medskip

The proof of Theorem \ref{thm:main} is complete.

\appendix

\section{The spherical and affine Coxeter diagrams}\label{classi_diagram}

For the reader’s convenience, we reproduce below the list of the irreducible spherical and irreducible affine Coxeter diagrams.

\newcommand{\scalee}{1.5}
\begin{table}[ht!]
\centering
\begin{minipage}[b]{7.7cm}
\centering
\begin{tikzpicture}[thick,scale=0.55, every node/.style={transform shape}]
\node[draw,circle] (A1) at (0,0) {};
\node[draw,circle,right=.8cm of A1] (A2) {};
\node[draw,circle,right=.8cm of A2] (A3) {};
\node[draw,circle,right=1cm of A3] (A4) {};
\node[draw,circle,right=.8cm of A4] (A5) {};
\node[left=.8cm of A1,scale=\scalee] {$\Huge A_n\ (n\geq 1)$};

\draw (A1) -- (A2)  node[above,midway] {};
\draw (A2) -- (A3)  node[above,midway] {};
\draw[loosely dotted,thick] (A3) -- (A4) node[] {};
\draw (A4) -- (A5) node[above,midway] {};


\node[draw,circle,below=1.2cm of A1] (B1) {};
\node[draw,circle,right=.8cm of B1] (B2) {};
\node[draw,circle,right=.8cm of B2] (B3) {};
\node[draw,circle,right=1cm of B3] (B4) {};
\node[draw,circle,right=.8cm of B4] (B5) {};
\node[left=.8cm of B1,scale=\scalee] {$B_n\ (n\geq 2)$};

\draw (B1) -- (B2)  node[above,midway] {$4$};
\draw (B2) -- (B3)  node[above,midway] {};
\draw[loosely dotted,thick] (B3) -- (B4) node[] {};
\draw (B4) -- (B5) node[above,midway] {};


\node[draw,circle,below=1.5cm of B1] (D1) {};
\node[draw,circle,right=.8cm of D1] (D2) {};
\node[draw,circle,right=1cm of D2] (D3) {};
\node[draw,circle,right=.8cm of D3] (D4) {};
\node[draw,circle, above right=.8cm of D4] (D5) {};
\node[draw,circle,below right=.8cm of D4] (D6) {};
\node[left=.8cm of D1,scale=\scalee] {$D_n\ (n\geq 4)$};

\draw (D1) -- (D2)  node[above,midway] {};
\draw[loosely dotted] (D2) -- (D3);
\draw (D3) -- (D4) node[above,midway] {};
\draw (D4) -- (D5) node[above,midway] {};
\draw (D4) -- (D6) node[below,midway] {};


\node[draw,circle,below=1.2cm of D1] (I1) {};
\node[draw,circle,right=.8cm of I1] (I2) {};
\node[left=.8cm of I1,scale=\scalee] {$I_2(p)\ (p\geq 5)$};

\draw (I1) -- (I2)  node[above,midway] {$p$};


\node[draw,circle,below=1.2cm of I1] (H1) {};
\node[draw,circle,right=.8cm of H1] (H2) {};
\node[draw,circle,right=.8cm of H2] (H3) {};
\node[left=.8cm of H1,scale=\scalee] {$H_3$};

\draw (H1) -- (H2)  node[above,midway] {$5$};
\draw (H2) -- (H3)  node[above,midway] {};


\node[draw,circle,below=1.2cm of H1] (HH1) {};
\node[draw,circle,right=.8cm of HH1] (HH2) {};
\node[draw,circle,right=.8cm of HH2] (HH3) {};
\node[draw,circle,right=.8cm of HH3] (HH4) {};
\node[left=.8cm of HH1,scale=\scalee] {$H_4$};

\draw (HH1) -- (HH2)  node[above,midway] {$5$};
\draw (HH2) -- (HH3)  node[above,midway] {};
\draw (HH3) -- (HH4)  node[above,midway] {};


\node[draw,circle,below=1.2cm of HH1] (F1) {};
\node[draw,circle,right=.8cm of F1] (F2) {};
\node[draw,circle,right=.8cm of F2] (F3) {};
\node[draw,circle,right=.8cm of F3] (F4) {};
\node[left=.8cm of F1,scale=\scalee] {$F_4$};

\draw (F1) -- (F2)  node[above,midway] {};
\draw (F2) -- (F3)  node[above,midway] {$4$};
\draw (F3) -- (F4)  node[above,midway] {};


\node[draw,circle,below=1.2cm of F1] (E1) {};
\node[draw,circle,right=.8cm of E1] (E2) {};
\node[draw,circle,right=.8cm of E2] (E3) {};
\node[draw,circle,right=.8cm of E3] (E4) {};
\node[draw,circle,right=.8cm of E4] (E5) {};
\node[draw,circle,below=.8cm of E3] (EA) {};
\node[left=.8cm of E1,scale=\scalee] {$E_6$};

\draw (E1) -- (E2)  node[above,midway] {};
\draw (E2) -- (E3)  node[above,midway] {};
\draw (E3) -- (E4)  node[above,midway] {};
\draw (E4) -- (E5)  node[above,midway] {};
\draw (E3) -- (EA)  node[left,midway] {};


\node[draw,circle,below=1.8cm of E1] (EE1) {};
\node[draw,circle,right=.8cm of EE1] (EE2) {};
\node[draw,circle,right=.8cm of EE2] (EE3) {};
\node[draw,circle,right=.8cm of EE3] (EE4) {};
\node[draw,circle,right=.8cm of EE4] (EE5) {};
\node[draw,circle,right=.8cm of EE5] (EE6) {};
\node[draw,circle,below=.8cm of EE3] (EEA) {};
\node[left=.8cm of EE1,scale=\scalee] {$E_7$};

\draw (EE1) -- (EE2)  node[above,midway] {};
\draw (EE2) -- (EE3)  node[above,midway] {};
\draw (EE3) -- (EE4)  node[above,midway] {};
\draw (EE4) -- (EE5)  node[above,midway] {};
\draw (EE5) -- (EE6)  node[above,midway] {};
\draw (EE3) -- (EEA)  node[left,midway] {};


\node[draw,circle,below=1.8cm of EE1] (EEE1) {};
\node[draw,circle,right=.8cm of EEE1] (EEE2) {};
\node[draw,circle,right=.8cm of EEE2] (EEE3) {};
\node[draw,circle,right=.8cm of EEE3] (EEE4) {};
\node[draw,circle,right=.8cm of EEE4] (EEE5) {};
\node[draw,circle,right=.8cm of EEE5] (EEE6) {};
\node[draw,circle,right=.8cm of EEE6] (EEE7) {};
\node[draw,circle,below=.8cm of EEE3] (EEEA) {};
\node[left=.8cm of EEE1,scale=\scalee] {$E_8$};

\draw (EEE1) -- (EEE2)  node[above,midway] {};
\draw (EEE2) -- (EEE3)  node[above,midway] {};
\draw (EEE3) -- (EEE4)  node[above,midway] {};
\draw (EEE4) -- (EEE5)  node[above,midway] {};
\draw (EEE5) -- (EEE6)  node[above,midway] {};
\draw (EEE6) -- (EEE7)  node[above,midway] {};
\draw (EEE3) -- (EEEA)  node[left,midway] {};

\draw (0,-18.5) node[]{} ;
\end{tikzpicture}
\caption{The irreducible spherical Coxeter diagrams}
\label{table:spheri_diag}
\end{minipage}
\begin{minipage}[t]{7.7cm}
\centering
\begin{tikzpicture}[thick,scale=0.55, every node/.style={transform shape}]
\node[draw,circle] (A1) at (0,0) {};
\node[draw,circle,above right=.8cm of A1] (A2) {};
\node[draw,circle,right=.8cm of A2] (A3) {};
\node[draw,circle,right=.8cm of A3] (A4) {};
\node[draw,circle,right=.8cm of A4] (A5) {};
\node[draw,circle,below right=.8cm of A5] (A6) {};
\node[draw,circle,below left=.8cm of A6] (A7) {};
\node[draw,circle,left=.8cm of A7] (A8) {};
\node[draw,circle,left=.8cm of A8] (A9) {};
\node[draw,circle,left=.8cm of A9] (A10) {};

\node[left=.8cm of A1,scale=\scalee] {$\widetilde{A}_n\ (n\geq 2)$};

\draw (A1) -- (A2)  node[above,midway] {};
\draw (A2) -- (A3)  node[above,midway] {};
\draw (A3) -- (A4) node[] {};
\draw (A4) -- (A5) node[above,midway] {};
\draw (A5) -- (A6) node[] {};
\draw (A6) -- (A7) node[] {};
\draw (A7) -- (A8) node[] {};
\draw[loosely dotted,thick] (A8) -- (A9) node[] {};
\draw (A9) -- (A10) node[] {};
\draw (A10) -- (A1) node[] {};


\node[draw,circle,below=1.7cm of A1] (B1) {};
\node[draw,circle,right=.8cm of B1] (B2) {};
\node[draw,circle,right=.8cm of B2] (B3) {};
\node[draw,circle,right=1cm of B3] (B4) {};
\node[draw,circle,right=.8cm of B4] (B5) {};
\node[draw,circle,above right=.8cm of B5] (B6) {};
\node[draw,circle,below right=.8cm of B5] (B7) {};
\node[left=.8cm of B1,scale=\scalee] {$\widetilde{B}_n\ (n\geq 3)$};

\draw (B1) -- (B2)  node[above,midway] {$4$};
\draw (B2) -- (B3)  node[above,midway] {};
\draw[loosely dotted,thick] (B3) -- (B4) node[] {};
\draw (B4) -- (B5) node[above,midway] {};
\draw (B5) -- (B6) node[above,midway] {};
\draw (B5) -- (B7) node[above,midway] {};


\node[draw,circle,below=1.5cm of B1] (C1) {};
\node[draw,circle,right=.8cm of C1] (C2) {};
\node[draw,circle,right=.8cm of C2] (C3) {};
\node[draw,circle,right=1cm of C3] (C4) {};
\node[draw,circle,right=.8cm of C4] (C5) {};
\node[left=.8cm of C1,scale=\scalee] (CCC) {$\widetilde{C}_n\ (n\geq 3)$};

\draw (C1) -- (C2)  node[above,midway] {$4$};
\draw (C2) -- (C3)  node[above,midway] {};
\draw[loosely dotted,thick] (C3) -- (C4) node[] {};
\draw (C4) -- (C5) node[above,midway] {$4$};


\node[draw,circle,below=1cm of C1] (D1) {};
\node[draw,circle,below right=0.8cm of D1] (D3) {};
\node[draw,circle,below left=0.8cm of D3] (D2) {};
\node[draw,circle,right=.8cm of D3] (DA) {};
\node[draw,circle,right=1cm of DA] (DB) {};
\node[draw,circle,right=.8cm of DB] (D4) {};
\node[draw,circle, above right=.8cm of D4] (D5) {};
\node[draw,circle,below right=.8cm of D4] (D6) {};
\node[below=1.5cm of CCC,scale=\scalee] {$\widetilde{D}_n\ (n\geq 4)$};

\draw (D1) -- (D3)  node[above,midway] {};
\draw (D2) -- (D3)  node[above,midway] {};
\draw (D3) -- (DA) node[above,midway] {};
\draw[loosely dotted] (DA) -- (DB);
\draw (D4) -- (DB) node[above,midway] {};
\draw (D4) -- (D5) node[above,midway] {};
\draw (D4) -- (D6) node[below,midway] {};


\node[draw,circle,below=2.5cm of D1] (I1) {};
\node[draw,circle,right=.8cm of I1] (I2) {};
\node[left=.8cm of I1,scale=\scalee] {$\widetilde{A}_1$};

\draw (I1) -- (I2)  node[above,midway] {$\infty$};


\node[draw,circle,below=1.2cm of I1] (H1) {};
\node[draw,circle,right=.8cm of H1] (H2) {};
\node[draw,circle,right=.8cm of H2] (H3) {};
\node[left=.8cm of H1,scale=\scalee] {$\widetilde{B}_2=\widetilde{C}_2$};

\draw (H1) -- (H2)  node[above,midway] {$4$};
\draw (H2) -- (H3)  node[above,midway] {$4$};


\node[draw,circle,below=1.2cm of H1] (HH1) {};
\node[draw,circle,right=.8cm of HH1] (HH2) {};
\node[draw,circle,right=.8cm of HH2] (HH3) {};
\node[left=.8cm of HH1,scale=\scalee] {$\widetilde{G}_2$};

\draw (HH1) -- (HH2)  node[above,midway] {$6$};
\draw (HH2) -- (HH3)  node[above,midway] {};


\node[draw,circle,below=1.2cm of HH1] (F1) {};
\node[draw,circle,right=.8cm of F1] (F2) {};
\node[draw,circle,right=.8cm of F2] (F3) {};
\node[draw,circle,right=.8cm of F3] (F4) {};
\node[draw,circle,right=.8cm of F4] (F5) {};
\node[left=.8cm of F1,scale=\scalee] {$\widetilde{F}_4$};

\draw (F1) -- (F2)  node[above,midway] {};
\draw (F2) -- (F3)  node[above,midway] {$4$};
\draw (F3) -- (F4)  node[above,midway] {};
\draw (F4) -- (F5)  node[above,midway] {};


\node[draw,circle,below=1.2cm of F1] (E1) {};
\node[draw,circle,right=.8cm of E1] (E2) {};
\node[draw,circle,right=.8cm of E2] (E3) {};
\node[draw,circle,right=.8cm of E3] (E4) {};
\node[draw,circle,right=.8cm of E4] (E5) {};
\node[draw,circle,below=.8cm of E3] (EA) {};
\node[draw,circle,below=.8cm of EA] (EB) {};
\node[left=.8cm of E1,scale=\scalee] {$\widetilde{E}_6$};

\draw (E1) -- (E2)  node[above,midway] {};
\draw (E2) -- (E3)  node[above,midway] {};
\draw (E3) -- (E4)  node[above,midway] {};
\draw (E4) -- (E5)  node[above,midway] {};
\draw (E3) -- (EA)  node[left,midway] {};
\draw (EB) -- (EA)  node[left,midway] {};


\node[draw,circle,below=3cm of E1] (EE1) {};
\node[draw,circle,right=.8cm of EE1] (EEB) {};
\node[draw,circle,right=.8cm of EEB] (EE2) {};
\node[draw,circle,right=.8cm of EE2] (EE3) {};
\node[draw,circle,right=.8cm of EE3] (EE4) {};
\node[draw,circle,right=.8cm of EE4] (EE5) {};
\node[draw,circle,right=.8cm of EE5] (EE6) {};
\node[draw,circle,below=.8cm of EE3] (EEA) {};
\node[left=.8cm of EE1,scale=\scalee] {$\widetilde{E}_7$};

\draw (EE1) -- (EEB)  node[above,midway] {};
\draw (EE2) -- (EEB)  node[above,midway] {};
\draw (EE2) -- (EE3)  node[above,midway] {};
\draw (EE3) -- (EE4)  node[above,midway] {};
\draw (EE4) -- (EE5)  node[above,midway] {};
\draw (EE5) -- (EE6)  node[above,midway] {};
\draw (EE3) -- (EEA)  node[left,midway] {};


\node[draw,circle,below=1.8cm of EE1] (EEE1) {};
\node[draw,circle,right=.8cm of EEE1] (EEE2) {};
\node[draw,circle,right=.8cm of EEE2] (EEE3) {};
\node[draw,circle,right=.8cm of EEE3] (EEE4) {};
\node[draw,circle,right=.8cm of EEE4] (EEE5) {};
\node[draw,circle,right=.8cm of EEE5] (EEE6) {};
\node[draw,circle,right=.8cm of EEE6] (EEE7) {};
\node[draw,circle,right=.8cm of EEE7] (EEE8) {};
\node[draw,circle,below=.8cm of EEE3] (EEEA) {};
\node[left=.8cm of EEE1,scale=\scalee] {$\widetilde{E}_8$};

\draw (EEE1) -- (EEE2)  node[above,midway] {};
\draw (EEE2) -- (EEE3)  node[above,midway] {};
\draw (EEE3) -- (EEE4)  node[above,midway] {};
\draw (EEE4) -- (EEE5)  node[above,midway] {};
\draw (EEE5) -- (EEE6)  node[above,midway] {};
\draw (EEE6) -- (EEE7)  node[above,midway] {};
\draw (EEE8) -- (EEE7)  node[above,midway] {};
\draw (EEE3) -- (EEEA)  node[left,midway] {};

\draw (0,-22.5) node[]{} ;
\end{tikzpicture}
\caption{The irreducible affine Coxeter diagrams}
\label{table:affi_diag}
\end{minipage}
\end{table}

\clearpage


\end{document}